\documentclass[12pt]{article}
\usepackage{amsmath, amssymb, amsthm}
\usepackage{hyperref}
\usepackage[latin1]{inputenc}
\usepackage{graphicx}
\usepackage{subfigure}
\usepackage{fullpage}

\newcommand{\T}{\mathbb{T}}
\newcommand{\R}{\mathbb{R}}
\newcommand{\Z}{\mathbb{Z}}

\newcommand{\Ss}{\mathbf{S}}

\newcommand{\Rev}{\mathcal{R}}
\newcommand{\Revh}{\widehat{\Rev}}
\newcommand{\Revt}{\widetilde{\Rev}}
\newcommand{\Cc}{\mathcal{C}}

\newcommand{\Splay}{\Theta_{\mathrm{splay}}}
\newcommand{\OmSplay}{\Omega_{\mathrm{splay}}}
\newcommand{\Sync}{\Theta_{\mathrm{sync}}}
\newcommand{\OmSync}{\Omega_{\mathrm{sync}}}
\newcommand{\fix}{\mathrm{Fix}}

\newcommand{\rset}[2]{\left\lbrace\, #1\,\left|\;#2\right.\right\rbrace}

\newcommand{\set}[2]{\rset{#1}{#2}}
\newcommand{\tset}[2]{\big\lbrace #1\,\big|\;#2\big\rbrace}
\newcommand{\sset}[1]{\left\lbrace #1\right\rbrace}

\DeclareMathOperator{\Fix}{Fix}

\newtheorem{prop}{Proposition}
\newtheorem{lem}{Lemma}

\newtheorem{cor}{Corollary}
\theoremstyle{definition}

\theoremstyle{remark}

\author{Peter Ashwin\thanks{Department of Mathematics, Centre for Predictive Modelling in Healthcare and Centre for Systems Dynamics and Control, University of Exeter, Exeter EX4 4QF, UK}~\thanks{EPSRC Centre for Predictive Modelling in Healthcare, University of Exeter, Exeter EX4 4QJ, UK}, Christian Bick${}^*$ and Oleksandr Burylko\thanks{Institute of Mathematics, National Academy of Sciences, 01601 Kyiv, Ukraine}}

\title{Identical phase oscillator networks: bifurcations, symmetry and reversibility for generalized coupling}
 
\begin{document}

\maketitle

\begin{abstract}
For a system of coupled identical phase oscillators with full permutation symmetry, any broken symmetries in dynamical behaviour must come from spontaneous symmetry breaking, i.e. from the nonlinear dynamics of the system. The dynamics of phase differences for such a system depends only on the coupling (phase interaction) function~$g(\varphi)$ and the number of oscillators $N$. This paper briefly reviews some results for such systems in the case of general coupling $g$ before exploring two cases in detail: (a)~general two harmonic form: $g(\varphi)=q\sin(\varphi-\alpha)+r\sin(2\varphi-\beta)$ and $N$ small (b)~the coupling $g$ is odd or even. We extend previously published bifurcation analyses to the general two harmonic case, and show for even $g$ that the dynamics of phase differences has a number of time-reversal symmetries. For the case of even $g$ with one harmonic it is known the system has $N-2$  constants of the motion. This is true for $N=4$ and any $g$, while for $N=4$ and more than two harmonics in $g$, we show the system must have fewer independent constants of the motion.
\end{abstract}

\section{Introduction}

The last 30 years has seen a lot of progress in understanding the collective behaviour of coupled oscillating dynamical systems, especially since the pioneering work of Winfree~\cite{Winfree} and the coupled phase oscillators of Kuramoto~\cite{kuramoto,Kuramoto1975} which has inspired many studies of criteria for synchronization and applications in a variety of areas, notably neuroscience \cite{ACN15}; see also~\cite{Acebron2005,Strogatz2000} for reviews. Most of the theoretical work has focussed on the continuum limit of large populations of various types of oscillator, and has been guided by large numerical simulations - see for example~\cite{Politi}. 

Nonetheless, relatively small networks may show a lot of nontrivial and remarkable behaviour that will be present in larger networks. We consider a general system of~$N$ coupled {\em identical} phase oscillators as in \cite{kuramoto} but with a general coupling/phase interaction function $g(\varphi)$:
\begin{equation}
\frac{d}{dt} \theta_k = F_k(\theta):= \omega + \frac{1}{N} \sum_{j=1}^N g(\theta_k-\theta_j)
\label{eq:oscN}
\end{equation}
where $(\theta_1,\ldots,\theta_N)\in \T^N$ represents the phase of the oscillators, $\omega$ the natural frequency of the oscillators and $g(\varphi)$ is the $2\pi$-periodic smooth coupling function (also called the phase interaction function). Such systems arise naturally as approximations of coupled limit cycle oscillators in a weak coupling limit; see for example \cite{AR15}.

\subsection{Identical coupled oscillators and symmetric dynamics} 

This case of identical oscillators is clearly highly symmetric - indeed, it is equivariant under the group of permutation symmetries $\Ss_N$ acting by permutation of coordinates as well as being equivariant under translation along the diagonal $(1,\dotsc,1)$ - giving a symmetry group $\Gamma=\Ss_N\times\T$. Some consequences of these symmetries have been studied in \cite{AS92} and subsequently in \cite{Brown2003} - tools from symmetric dynamics \cite{GolSte} allow one to understand a lot of the behaviour of the system by looking at the action of the symmetry group on phase space. In particular, for any point in $\theta\in\T^N$ we associate the {\em isotropy subgroup} which is the subgroup $\set{g\in\Gamma}{g\theta=\theta}\subset \Gamma$. For each subgroup $G\subset \Gamma$ we can define the {\em fixed point subspace} $\fix(G)=\set{\theta}{g\theta=\theta~\forall ~g\in G}$; it is a standard result from \cite{GolSte} that for any $G$, $\fix(G)$ is dynamically invariant for any $\Gamma$-equivariant dynamical system. There is a twist in the story in that here~$\Gamma$ acts on the manifold $\T^N$ rather than $\R^N$, giving rise to more non-trivial structures than expected from~\cite{GolSte}. In this paper we also explore the presence of time-reversal symmetries \cite{Sevryuk1986,Lamb1997} for even $g$.

In order to understand how (\ref{eq:oscN}) may have complex spatio-temporal attractors of various symmetry types, we study this fully symmetric case for small $N$ and specific $g$ to gain an insight into the general case. For the case of Kuramoto--Sakaguchi coupling \cite{SK86} 
\begin{equation}
g(\varphi)=-\sin(\varphi-\alpha)
\label{eq:KS}
\end{equation}
one can perform a detailed analysis of the general $N$ case and one finds either synchrony or antisynchrony are attracting, with the cases $\alpha=\pi/2$, $3\pi/2$ giving special integrable structure~\cite{Watanabe1994}  that corresponds to ``vertical branches'' at bifurcation. However, this sort of behaviour is highly degenerate for a general dissipative system. On addition of higher harmonics to the coupling function we expect, and usually find, that the degeneracies typically unfold into generic situations. Indeed such higher harmonic terms are present in a generic system even near Hopf bifurcation \cite{K15,AR15}. A range of solutions, their stabilities and bifurcations, and the invariant subspaces imposed on phase space for the case of general~$g$ is discussed in~\cite{AS92}. This paper briefly reviews and then extends and applies some of these results in two contexts.

The first aim of the paper is to consider the detailed bifurcation behaviour for the most generic two harmonic coupling function. This means we determine which degeneracies of~\eqref{eq:KS} are unfolded by addition of the second harmonic, for the cases of $N\leq4$. In doing so we extend the analysis in \cite{ABM08} to the most general two harmonic coupling. The paper \cite{ABM08} examines the cases $N=3$ and $4$ for (\ref{eq:oscN}) and presents complete global bifurcation diagrams for the phase interaction function discussed in \cite{HMM93}
\begin{equation}
g(\varphi)=-\sin(\varphi+\alpha)+r \sin 2\varphi.
\label{eq:hmm}
\end{equation}
However, for this parametrized family the only even coupling functions are already in (\ref{eq:KS}). To next order \cite{AR15}, the generic lowest order terms will be of the form 
\begin{equation}
g(\varphi)=q\sin(\varphi-\alpha)+r \sin (2\varphi-\beta)
\label{eq:g2}
\end{equation}
where $q,r,\alpha$ and $\beta$ are arbitrary constants (specific cases of such coupling was considered by~\cite{AOWT07} in the context of $N=5$). More generally one can consider as in~\cite{Daido} $L$-harmonic coupling
\begin{equation}
g(\varphi)=\sum_{j=1}^{L}q_j\sin(j\varphi-\alpha_j)
\label{eq:gLharm}
\end{equation}
where, without loss of generality, we can assume $q_1q_L\neq 0$, and (\ref{eq:g2}) corresponds to $L=2$.

The second aim of this paper is to consider the behaviour of the system (\ref{eq:oscN}) for arbitrary, but even $g(\varphi)=g(-\varphi)$. Following \cite{Bick2012b}, we note there will be special structure in the dynamics, and a range of non-trivial recurrent dynamics appears for $g$ that are small perturbations of even functions. This has been illustrated by the observation of chaotic attractors \cite{Bick2011} and extreme sensitivity to detuning \cite{ABM08} for fully symmetric systems of oscillators. Indeed, chimera states \cite{Abrams2004,Panaggio2015,Laing2009} for structured (but not necessarily fully symmetric) networks are usually found close to even coupling.
We observe that there is zero divergence for (\ref{eq:oscN}) and time-reversal symmetries for the system of phase difference (though typically not for the original system). It is known there are $N-2$ independent smooth integrals of the motion for any $N\geq 3$ for the special case $g(\varphi)=\cos(\varphi)$ \cite{Watanabe1994}. This still holds for general even $g$ and $N=3$ but, surprisingly, we show here that it does not hold for $N=4$ and general even $g$: the flow is not fully integrable. 

The paper is structured as follows: after discussing some parameter and phase symmetries, we recall the reduction of the dynamics to a canonical invariant region. In Section~\ref{sec:bifs} we summarise some stability properties and bifurcations of periodic orbits by symmetry type. Section~\ref{sec:N34} applies this to general two harmonic coupling and $N=3$ or $4$ oscillators, extending the bifurcation analysis in \cite{ABM08} to general second harmonic coupling. This includes a new characterization of bifurcation of two-cluster states as the discriminant of a certain polynomial equation. Section~\ref{sec:OddEven} examines some consequences of the coupling function $g$ itself being odd, or even for small $N$. If $g$ is odd, it is known that the system for phase differences (but not the original system) has a variational structure and can be written as a gradient flow. If $g$ is even, there are additional time-reversal symmetries \cite{Lamb1997} for the system for phase differences, as well as a zero divergence condition. We characterise the points with nontrivial time-reversal symmetries for $N=3$ and $N=4$ in detail. Finally, Section~\ref{sec:Discuss} considers the relation between even coupling and integrability. For $N=3$ and general even $g$, or for $N\geq 4$ and $g(\varphi)=\cos(\varphi)$ it is known since \cite{Watanabe1994} that there are $N-2$ independent integrals of the motion. This constrains the dynamics for $N\geq 4$ to be at most two-frequency quasiperiodic. We show this is not the case for $N\geq 4$ and more general even $g$: for examples with $N=4$ and $g$ with $L\leq 4$ harmonics, we find structures that respect the time-reversal symmetries, but that cannot appear in a fully integrable structure. We discuss this, and highlight some open questions.

\subsection{Parameter symmetries for two harmonic coupling}

If we consider general two harmonic coupling of the form (\ref{eq:g2}), the case (\ref{eq:hmm}) considered in \cite{ABM08} corresponds to $q=-1$, $\beta=0$. We wish to understand the behaviour of (\ref{eq:oscN},\ref{eq:g2}) on varying the parameters
\[(q,r,\alpha,\beta)\in \R^2\times[0,2\pi)^2.\]
There are parameter symmetries (such that $g(\varphi)\mapsto g(\varphi)$ for all $\varphi$)
\begin{align}
(q,r,\alpha,\beta)&\mapsto(-q,r,\alpha+\pi,\beta),\label{eq:parsym1}\\
(q,r,\alpha,\beta)&\mapsto(q,-r,\alpha,\beta+\pi).\label{eq:parsym2}
\end{align}
Using (\ref{eq:parsym1}) and rescaling time by a positive scaling, we assume from hereon that $q=-1$: using (\ref{eq:parsym2}) we can assume from hereon that $r\geq 0$. We choose $q=-1$ to give easy comparability to \cite{ABM08}. There is a time-reversal parameter symmetry (such that $g(\varphi)\mapsto -g(\varphi)$ for all $\varphi$), assuming $q=-1$ this is
\begin{align}
(r,\alpha,\beta)=(r,\alpha+\pi,\beta+\pi)
\label{eq:timerev}
\end{align}
By classifying qualitative dynamics up to time-reversal, we will can reduce to parameters in the region
\begin{equation}
q=-1,~~(r,\alpha,\beta)\in [0,\infty)\times[0,2\pi)\times [0,\pi)
\label{eq:qrab}
\end{equation}
in the later bifurcation analyses; using (\ref{eq:timerev}) we can classify for $\beta\in[\pi,2\pi)$. Finally, there is a phase-reversal parameter symmetry (i.e. such that $g(\varphi)\mapsto -g(-\varphi)$ for all $\varphi$), assuming $q=-1$ this is
\begin{align}
(r,\alpha,\beta)=(r,-\alpha,-\beta).
\end{align}
This means in principle we can reduce to (\ref{eq:qrab}) with $\beta\in[0,\pi/2]$ and apply phase and time reversal as necessary.

\subsection{Phase symmetries: the canonical invariant region}

Because the dynamics of (\ref{eq:oscN}) depends only on phase differences, one can usefully reduce the dynamics from a flow on $\T^N$ to one on $\T^{N-1}$ which represents only the phase differences. This corresponds to looking at the quotient dynamics of the system on the $\T$-orbits. More precisely, if we consider a projection $\Pi:\T^N\rightarrow \T^{N-1}$ that maps the $\T$-orbits onto points and consider $\tilde{\theta}=\Pi(\theta)\in\T^{N-1}$, the original system (\ref{eq:oscN}) becomes a reduced system on $\T^{N-1}$ of the form
\begin{equation}
\frac{d}{dt} \tilde{\theta} = \tilde{F}(\tilde{\theta}).
\label{eq:oscNreduced}
\end{equation}
This clearly depends on $g$ and $N$ but not on $\omega$.

There are several ways to choose $\Pi$, or, equivalently, a set of~$N-1$ phase differences $\theta_j-\theta_k$ that span the set of all phase differences. However, permutation symmetry will never act orthogonally on such a set of phase differences if $N\geq 3$, meaning that symmetries are not obvious from plots of phase differences. For $N=3$ and $N=4$ there is a useful representation introduced by~\cite{AS92} that avoids this problem. If we examine an orthogonal projection of the dynamics onto a codimension one subspace orthogonal to the diagonal $(1,\dotsc,1)$, this will carry an orthogonal (and irreducible) representation of~$\Ss_{N}$ on $\R^{N-1}$, the covering space of~$\T^{N-1}$. In addition, the complement of fixed point subspaces where two of the phases are identical (isotropy~$\Ss_2$) forms a partition of~$\R^{N-1}$ into connected components that are all images of the {\em canonical invariant region} (or CIR) \cite{AS92}. This is the set
\begin{equation}
\Cc = \set{\theta\in\T^N}{0=\theta_1<\theta_2<\dotsb<\theta_N<2\pi}
\end{equation}
whose boundary consists of points with~$\Ss_2$ isotropy. There is a residual symmetry $\Z_N=\Z/N\Z$ $\langle \tau\rangle$ on~$\Cc$ generated by
\begin{equation}\label{eq:RemSym}
\tau: (0, \theta_2, \dotsc, \theta_N)\mapsto(0, \theta_3-\theta_2, \dotsc, \theta_N-\theta_2, 2\pi-\theta_2).
\end{equation}
Note that~$\Splay = (0, \theta_2, \dotsc, \theta_N)\in\Cc$ with $\theta_k=\theta_{k-1}+\frac{2\pi}{N}$ is the only fixed point of the action of~$\tau$ in~$\Cc$~\cite{AS92}.

Figure~\ref{fig:CIR} illustrates the $\Cc$ for the cases $N=3$ and $N=4$ - the boundaries of $\Cc$ are invariant for dynamics governed by (\ref{eq:oscN}) for any coupling function $g$, though they may not be invariant for arbitrarily small perturbations that make the $g$ or the $\omega$ different for different oscillators - in this case we may have {\em extreme sensitivity to detuning} \cite{ABM08}. Nonetheless, even in this case the trajectory will remain most of the time in one image of the CIR, and the dynamics can usefully be understood as being a perturbation of the highly symmetric case.

\begin{figure}%
\centerline{
\includegraphics[width=11cm]{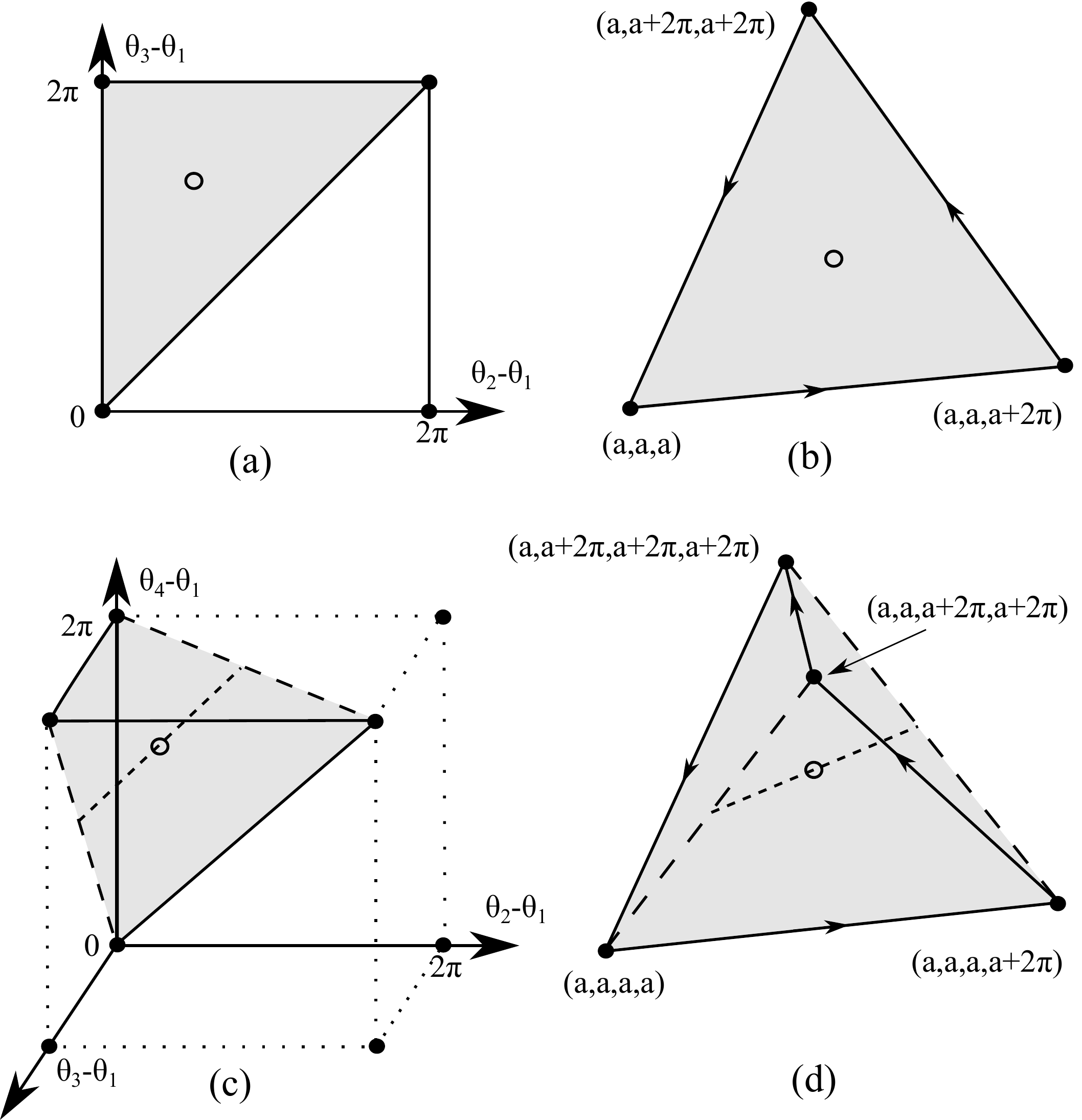}
}%
\caption{Structure of the canonical invariant region $\Cc$ for $N=3$ and $N=4$. Panels~(a,c) show~$\Cc$ (shaded) in terms of the phase differences $\theta_j-\theta_1$ and~(b,d) as an orthogonal projection of the diagonal into~$\R^2$ and~$\R^3$, respectively. The edges of~$\Cc$ for (a,b) and the faces of~$\Cc$ for (c,d) are points with~$\Ss_2$ isotropy. The filled circles represent different points on the lift that that correspond to fully synchrony on the torus. The open circle represents the splay phase solution in~$\Cc$. In (c,d) the solid lines have isotropy $\Ss_3\times\Ss_1$ while the long dashes have isotropy $\Ss_2\times\Ss_2$. The short dashed lines have isotropy $(\Ss_1)^2\times_S \Z_2$ - typical points being $(a,b,a+\pi,b+\pi)$. In each case there is a residual $\Z_{N-1}$ symmetry indicated by the arrows in (b,d) and $(N-1)!$ symmetric copies of $\Cc$ pack a generating region for the torus.}%
\label{fig:CIR}%
\end{figure}

Within~$\Cc$ there are representatives of points with all possible symmetry types. Particularly important are \emph{$\ell$-cluster states} which correspond to states where for some partition of $N=m_1+\cdots+m_\ell$ there are precisely $m_k$ oscillators at the same phase, for some $\ell\geq 2$ (the number of clusters) and $m_k\geq 1$ (the size of the $k$th cluster). These correspond to fixed point subspaces for isotropy groups $\Ss_{m_1}\times\cdots\times\Ss_{m_\ell}$. These are located on the boundary of $\Cc$ - there are also points with spatio-temporal symmetries $(\Ss_{m_1}\times\cdots\times\Ss_{m_\ell})^M\times_s\Z_M$ with $M>1$ and $N=M(m_1+\cdots+m_\ell)$; see \cite{AS92}. NB: To simplify notation we omit copies of~$\Ss_1$ when describing cluster states where this is clear from context: for example $\Ss_2$ for the case $N=4$ corresponds to a three cluster state with isotropy group $\Ss_2\times \Ss_1\times \Ss_1$.

\section{Stabilities and bifurcations of solutions forced by symmetries}
\label{sec:bifs}

We recall from \cite{AS92} a number of important synchronous solutions that are forced by symmetry, and their bifurcation properties. 

\subsection{Stability and bifurcation of synchronised solutions}

The {\em fully synchronous solution} (in-phase oscillation) of (\ref{eq:oscN}) is a periodic orbit given by
\begin{equation}
\Sync(t) =\left(\theta(t),\dotsc,\theta(t)\right)
\end{equation}
where $\theta(t)=\OmSync t+\theta_0$ and $\OmSync=\omega+g(0)$. 

\begin{lem}[\cite{AS92}]
The fully synchronous solution~$\Sync(t)$ is linearly stable if and only if 
\begin{equation}
g'(0)<0.
\end{equation}
\end{lem}

For general two harmonic coupling (\ref{eq:g2}) with $q=-1$, we have $g(0)=2 r\cos \beta-\cos \alpha$ and so there is bifurcation to loss of stability of this solution when
\begin{equation}
r= \frac{\cos \alpha}{2\cos \beta}
\label{eq:syncbif}
\end{equation}
independent of $N$.

A \emph{splay phase solution} (rotating wave) is, up to permutation, a periodic orbit of the form
\begin{equation}\label{eq:Splay}
\Splay(t)=\left(\theta(t),\theta(t)-\frac{2\pi }{N},\dots,\theta(t)-\frac{(N-1)2\pi}{N}\right)
\end{equation}
where $\theta(t)=\OmSplay t+\theta_0$. There are $(N-1)!$ different splay phase solutions from the possible permutations, but only one in the canonical invariant region $\Cc$. Let $i=\sqrt{-1}$.

\begin{lem}[\cite{AS92}]
The splay phase solutions are linearly stable if the real parts of 
\begin{equation}
\sum_{j=1}^N g'\left(\frac{2\pi}{N}j\right) (1-\rho^{pj})
\end{equation}
are positive for $p=1,\ldots N-1$, where $\rho=\exp(\frac{2\pi i}{N})$.
\end{lem}

Let $\eta_{j}=g'\left(\frac{2\pi}{N}j\right)$ and $\nu_{p}=\rho^p=\exp(\frac{2\pi i}{N}p)$. If $\Theta=\Splay$ then the matrix given by
\begin{align*}
A_{ii}(\Theta)&=\frac{1}{N}\sum_{j=1}^{N}g'(\Theta_{j-1}-\Theta_{i-1})=\frac{1}{N}\sum_{j=1}^{N}\eta_{j},\\
\intertext{for $i=1, \dotsc, N$ and}
A_{ij}(\Theta)&=-\frac{1}{N}g'(\Theta_{j-1}-\Theta_{i-1})=-\frac{1}{N}\eta_{j-i}
\end{align*}
for $i,j=1,\dots,N,\ j\ne i$ has eigenvalues
\begin{align*}
\lambda_{p}(\Theta)&=\frac{1}{N}\sum_{j=1}^{N}\eta_{j}(1-\nu_{p}^{j})=\frac{1}{N}\sum_{j=1}^{N-1}\eta_{j}(1-\nu_{p}^{j})\\
&=\frac{1}{N}\sum_{j=1}^{N-1}\eta_{j}\left(1-\cos\left(\frac{2pj\pi}{N}\right)\right)- i\frac{1}{N}\sum_{j=1}^{N-1}\eta_{j}\sin\left(\frac{2pj\pi}{N}\right),\quad p=1,\dots,N.
\end{align*}
Note that $\nu_{N}^{j}=\exp\left(\frac{2\pi i}{N}jN\right)=1$
implies $\lambda_{N}(\Theta)=0$ for any splay phases. The system has $[N/2]$ complex conjugate pairs 
\[
\lambda_{N-p}(\Theta)=\lambda_{-p}(\Theta)=\overline{\lambda}_{p}(\Theta),
\]
in the even--dimensional case $\lambda_{N/2}(\Theta)=\frac{1}{N}\sum_{j=1}^{N-1}g'\left(\frac{2\pi}{N}j\right)\left(1-(-1)^{j}\right)\in\mathbb{R}$. Splay states~$\Theta$ are stable when 
\[
\mbox{Re}(\lambda_{p}(\Theta))=\frac{1}{N}\sum_{j=1}^{N-1}g'\left(\frac{2\pi}{N}j\right)\left(1-\cos\left(\frac{2pj\pi}{N}\right)\right)<0,\quad p=1,\dots,N-1.
\]
The system for phase differences near the splay state has an Andronov--Hopf bifurcation (which can be degenerate or multiple) when $\mbox{Re}(\lambda_{p}(\Theta))=0$ for some index $p=1,\dots N-1$, $p\ne N/2$, and $\mbox{Im}(\lambda_{p}(\Theta))\ne 0$.

For general two harmonic coupling (\ref{eq:g2}) this can be expressed as conditions on the parameters $q,r,\alpha,\beta$, though note that unlike the synchronous state, this will depend on $N$ in a fairly complex way. The cases for $N=3$ and $N=4$ are discussed in Section~\ref{sec:N34}.

\subsection{Antiphase states: the set with zero order parameter.}

We discuss briefly a set that is a union of invariant manifolds for some forms of coupling or for small enough $N$. We note that it is not invariant for generalised coupling $N\geq 5$. Consider the Kuramoto order parameter for system (\ref{eq:oscN}) is
\[
R\exp( i\psi)=\frac{1}{N}\sum_{k=1}^{N}\exp( i\theta_{k})
\]
and define the set with zero order parameter $R=0$ by
\[
M^{(N)}=\set{(\theta_{1},\dots,\theta_{N})}{\sum_{k=1}^{N}\exp( i\theta_{k})=0}.
\]
This set $M^{(N)}$ is union of $(N-2)$-dimensional manifolds in $\mathbb{T}^{N}$ \cite{ABM08}; it can be characterised by the set where two variables are given in terms of the other $N-2$ variables by expressions
\begin{align*}
\theta_{N}&=\arctan\left(\frac{f_{2}}{f_{1}}\right)+\frac{1}{2}\arccos\left(\frac{f_{1}^{2}+f_{1}^{2}}{2}-1\right)+\frac{\pi}{2}(1-{\rm sign}(f_{1})),\\
\theta_{N-1}&=\arctan\left(\frac{f_{2}}{f_{1}}\right)-\frac{1}{2}\arccos\left(\frac{f_{1}^{2}+f_{1}^{2}}{2}-1\right)+\frac{\pi}{2}(1-{\rm sign}(f_{1})),
\end{align*}
where 
\begin{align*}
f_{1}&=-\sum_{j=2}^{N}\cos\theta_{j}, & f_{2}&=-\sum_{j=2}^{N}\sin\theta_{j}.
\end{align*}
The manifolds can be written in phase differences (for example $\varphi_{i}=\theta_{1}-\theta_{i}$, $i=1,\dots N-1$) is $(N-3)$--dimensional set in phase
space $\mathbb{T}^{N-1}$.

\begin{prop}[{\cite[Lemma 1]{BP08}}]
If the coupling function has only one harmonic, for any $N$ the set $M^{(N)}$ is dynamically invariant; moreover it consists of solutions with fixed phase differences of (\ref{eq:oscN}).
\end{prop}

However, although some solutions (in particular, splay states (\ref{eq:Splay})) belong to $M^{(N)}$ for arbitrary coupling function and arbitrary $g$, quite surprisingly it is not generally invariant for general coupling.
 
\begin{prop}
The set~$M^{(N)}$ is invariant for the system (\ref{eq:oscN}) for \emph{arbitrary} coupling function $g(\varphi)$ when $N=2,3,4$ ($\omega=0$), but not in general for $N\geq 5$.
\end{prop}

\begin{proof}
For $N=2$, the equation $\exp( i\theta_{1})+\exp( i\theta_{2})=0$ has solution $\theta_{2}=\theta_{1}+\pi$, while for $N=3$, the equation $\sum_{k=1}^{3}\exp( i\theta_{k})=0$ has solution $\theta_{2}=\theta_{1}+\frac{2\pi}{3}$, $\theta_{3}=\theta_{1}-\frac{2\pi}{3}$. Both cases correspond to the splay phase state which is a zero dimensional fixed point subspace for the phase differences - hence it is dynamically invariant. For $N=4$, the equation $\sum_{k=1}^{4}\exp( i\theta_{k})=0$ has solution
$\theta_{3}=\theta_{1}+\pi,\ \theta_{4}=\theta_{2}+\pi$. This corresponds to the fixed point subspace of $(\Ss_1\times \Ss_1)^2\times_s\Z_2$ which is also therefore invariant.

On the other hand, in the case $N\geq 5$ the corresponding sets $M^{(N)}$ are not fixed point subspaces, and therefore not necessarily invariant. For example, the set $M^{(5)}$ includes points of the form
\[\theta=\left(\theta_1,\theta_1+\frac{2\pi}{3},\theta_1+\frac{4\pi}{3
},\theta_4,\theta_4+\pi\right)\]
which is a two dimensional subspace on which
$$
R=e^{i\theta_1}(1+e^{2i \pi/3}+e^{4i \pi/3})+e^{i\theta_4}(1+e^{i\pi})=0.
$$
However, the smallest fixed point subspaces the contains all points of this form is the trivial subspace $\T^5$, meaning the symmetry does not force invariance of $M^{(5)}$.
\end{proof} 

Note that for $N=4$ coupled oscillators, the set \[M^{(4)}=\sset{(0, \varphi,\pi,\varphi+\pi),(0, \pi,\varphi,\varphi+\pi),(0, \varphi,\varphi+\pi,\pi)}\]
(one half-line in each invariant region) as a union of invariant manifolds that correspond to the states with isotropy $\Z_2$.

\subsection{Stability and bifurcation of two cluster states}
\label{sec:twocluster}

Consider any two-cluster states with isotropy $\Ss_{p}\times \Ss_{N-p}$
given by
\begin{equation}
\Theta_{p,q}(t)=\left(\theta_a(t),\dotsc,\theta_a(t),\theta_b(t),\dotsc,\theta_b(t)\right)
\end{equation}
such that $\theta_{i}=\theta_{a}(t)$ for $i=2,\dots,p$ and $\theta_{i}=\theta_{b}(t)$ for $i=p+1,\dots,N$. If we denote $\varphi_{p}=\theta_{a}-\theta_{b}$ then (see for example \cite{Oroszetal2009} for the general case) (\ref{eq:oscN}) gives
\[
\sum_{j=1}^{N}g(\theta_{i}-\theta_{j})=\begin{cases}
pg(0)+(N-p)g(\varphi_{p}), & i=1,\dots,p,\\
pg(-\varphi_{p})+(N-p)g(0), & i=p+1,\dots,N.
\end{cases}
\]
In this phase difference coordinate we can write
\begin{equation}
\frac{d\varphi_{p}}{dt}=pg(0)+(N-p)g(\varphi_{p})-pg(-\varphi_{p})-(N-p)g(0).\label{eq:Eq_on_twocluster_IM}
\end{equation}
Any fixed point in the phase differences will therefore satisfy
\begin{equation}\label{eq:twoclusters}
(2p-N)\left[g(0)-g_{+}(\varphi_{p})\right]+Ng_{-}(\varphi_{p})=0.
\end{equation}

For general two harmonic coupling (\ref{eq:g2}) this means that
\begin{align}
0 &=(2p-N)\left[g(0)-g_{+}(\varphi_{p})\right]+Ng_{-}(\varphi_{p})\nonumber\\
&= (2p-N)\left(\cos(\varphi_{p})-1\right)\left[q\sin\alpha+2r\sin\beta\left(\cos(\varphi_{p})+1\right)\right]\nonumber\\
&\qquad+N\sin(\varphi_{p})\left[q\cos\alpha+2r\cos(\varphi_{p})\cos\beta\right]=0
\end{align}
This has a trivial solution: the synchronous state with $\varphi_{p}=0$. Removing the factor of $\sin(\varphi_{p}/2)$ this means the nontrivial two-cluster states will have a phase difference $\varphi_{p}$ where
\begin{align}
0 &=  (N-2p)\sin(\varphi_{p}/2)\left[q\sin\alpha+4r\sin\beta\cos^{2}(\varphi_{p}/2)\right]\nonumber\\
&\qquad+N\cos(\varphi_{p}/2)\left[q\cos\alpha+2r\cos\beta\left(2\cos^{2}(\varphi_{p}/2)-1\right)\right].\label{eq:twocluster1}
\end{align}
Note that two cluster-states for $g(\varphi)=-\sin(\varphi-\alpha)$
(i.e. the special case $r=0$) are considered in \cite[eq.~(17)]{BP08}. By setting $\tau=\tan \varphi_{p}/2$, then (\ref{eq:twocluster1}) can be expressed as a cubic equation. If we write $p=NP$ so that $0<P<1$, this implies the bifurcations occur precisely when the cubic equation for $\tau$
\begin{equation}
(1-2P)[q\sin\alpha]\tau^3+[q\cos\alpha-2r\cos\beta]\tau^2+(1-2P)[q\sin\alpha+4r\sin\beta]\tau+[q\cos\alpha+2r\cos\beta]=0
\end{equation}
has discriminant zero, namely when
\begin{equation}
0= a_0 + a_1 P + a_2 P^2 + a_3 P^3 +a_4 P^4,
\label{eq:twocluster3}
\end{equation}
where $a_k$ are polynomials in $q$, $r$, and $\cos$ and $\sin$ of $\alpha$ and $\beta$. The expressions for the $a_k$ are given in Appendix~\ref{app:twocluster}.

There are also bifurcations with fixed points of two-cluster manifold that occur in transversal to this manifold direction: these create three-cluster states. For example, pitchfork bifurcations (bifurcation
line $ID$ in \cite[Fig 1]{ABM08} for three oscillators can occur transverse to the lines $\varphi_{1}=0$ (or $\varphi_{2}=0,$ $\varphi_{1}=\varphi_{2}$). Such bifurcations cannot be found by examination just of the equation on invariant manifold (\ref{eq:Eq_on_twocluster_IM}) (or (\ref{eq:twoclusters})) but need to consider the transverse eigenvalues of the solutions, expressions for which are given in \cite{Oroszetal2009}. Later on we use numerical path following to locate these.

\section{Bifurcation scenarios for low $N$ and two harmonic coupling}
\label{sec:N34}

We review and extend the exploration of parameter space of (\ref{eq:oscN}) with general two harmonic coupling (\ref{eq:g2}) from the special case $\beta=0$ \cite{ABM08} for $N=3$ and $N=4$ to more general $\beta\neq 0$.

\subsection{Bifurcations for $N=3$}

Fig.~\ref{fig:N=3betaneq0} illustrates the bifurcations that determine the structure of the phase space for $N=3$ in the case $q=-1$, on varying $\alpha\in[0,\pi)$ and $r\geq 0$, for four choices of $\beta$. Note that the case $\beta=0$ in the dot-dashed region corresponds to the case considered in \cite{ABM08}. Some of the curves are calculated analytically while the remainder of the curves on Figure~\ref{fig:N=3betaneq0} are obtained by numerical path following of bifurcations in the reduced system (\ref{eq:oscNreduced}) using XPPAUT \cite{xppaut}.

\begin{figure}%
\centerline{
\includegraphics[width=15cm]{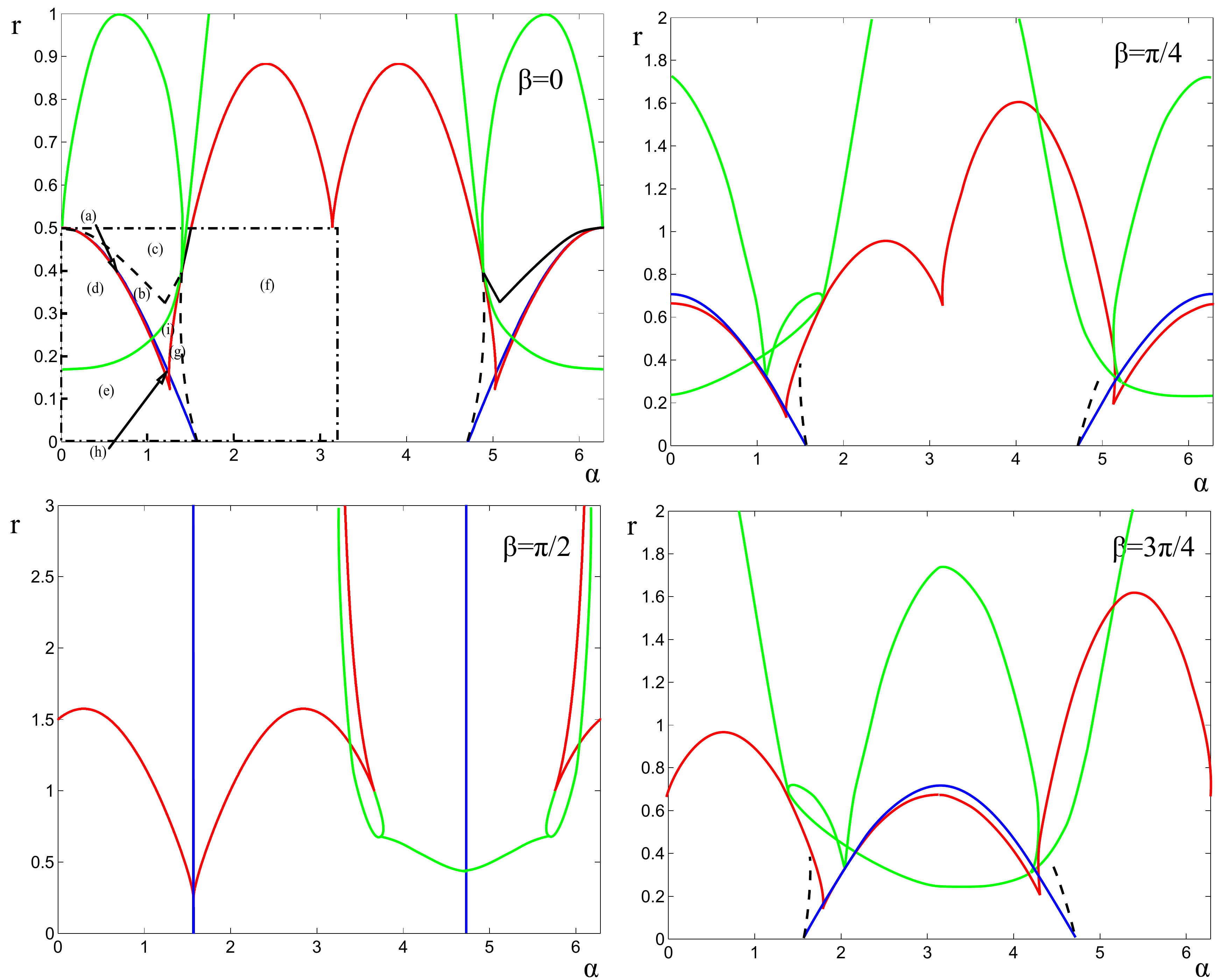}}%
\caption{Summary of bifurcations for $N=3$ on varying $\alpha$ and $r$ for four choices of $\beta$. The blue lines (\ref{eq:syncbif}) indicate (simultaneous) steady bifurcation of in-phase and Hopf bifurcation of splay phase solutions. The red lines (\ref{eq:twoclusterbifN3p1}) indicate saddle-node bifurcations of non-trivial $\Ss_2$ states. The green lines (found using XPPAUT) give the location of steady bifurcations involving equilibria where all phase differences are non-zero. The black dashed lines indicate homoclinic/heteroclinic bifurcations of periodic solutions. The letters (a-h) show the location of points corresponding to the phase planes shown in Figure~\ref{fig:N=3beta=0pp}. The dot-dashed regions for the case $\beta=0$ corresponds to the diagram detailed in \cite{ABM08}. Note that in the case $\beta=\pi/2$ and $\alpha=\pi/2$ or $3\pi/2$ the system is integrable, while the case $\beta=3\pi/4$ can be found from that for $\pi/4$ by the phase-and time-reversal symmetry $(r,\alpha,\beta)\mapsto(r,\pi-\alpha,\pi-\beta)$.}%
\label{fig:N=3betaneq0}%
\end{figure}

\begin{figure}%
\centerline{\includegraphics[width=13cm]{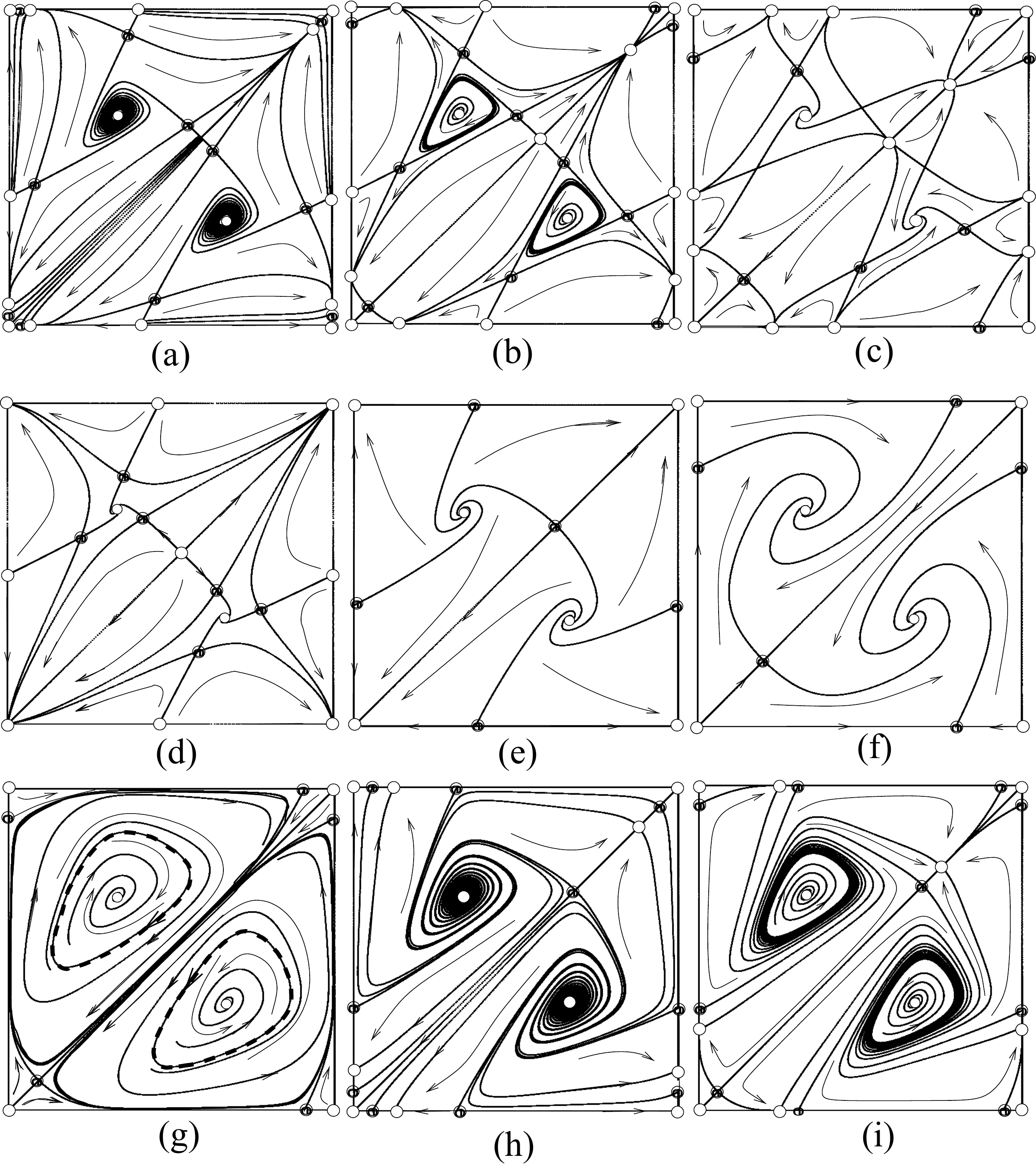}
}%
\caption{Phase portraits for $N=3$ and $\beta=0$, displayed as in Figure~\ref{fig:CIR}(a) on varying $\alpha$ and $r$, at locations indicated in Figure~\ref{fig:N=3betaneq0}. The transitions and codimension two points are described in detail in \cite{ABM08}. }%
\label{fig:N=3beta=0pp}%
\end{figure}

Note that from the previous section, we know that the in-phase solution undergoes bifurcation at (\ref{eq:syncbif}). The splay phase solution with $N=3$ the stability is governed by the eigenvalues
\begin{equation}
\lambda_{\mathrm{splay},N=3}= -\frac{1}{2} q\cos \alpha+ r \cos \beta \pm i \left[\frac{1}{2} q \sin \alpha-  r \sin \beta\right]
\label{eq:N3splay}
\end{equation}
meaning there is a Hopf bifurcation of the splay phase also at
\[
r=\frac{\cos \alpha}{2\cos \beta}.
\]
This is precisely the same location as (\ref{eq:syncbif}) - a non-generic phenomenon! For $N=3$ there is only one type of two-cluster state: this is for $p=1$. Setting $P=1/3$ and $q=-1$, (\ref{eq:twocluster1}) can be expressed as the cubic equation for $\tau=\tan(\varphi/2)$:
\begin{equation}
-\sin \left( \alpha \right) {\tau}^{3}
-3\left(2\,\cos \left( \beta \right) r+\,\cos \left( \alpha \right)\right) {\tau}^{2}
+\left(4\,\sin \left( \beta
\right) r-\sin \left( \alpha \right)\right) \tau+6\,\cos \left( \beta \right) 
r-3\,\cos \left( \alpha \right) =0
\end{equation}
which has bifurcations of solutions precisely when the discriminant is zero: this implies the bifurcations of two-cluster states occur when
\begin{align}
	0 & =576\left(9\cos^{2}\beta+\sin^{2}\beta\right)\cos^{2}\beta r^{4}\nonumber \\
	& \qquad+64(36\sin\alpha\cos^{2}\beta\sin\beta+9\cos\alpha\cos\beta\sin^{2}\beta\nonumber \\
	& \qquad+4\sin\alpha\sin^{3}\beta+81\cos\alpha\cos^{3}\beta)r^{3}\nonumber \\
	& \qquad+16(9\cos^{2}\alpha\sin^{2}\beta-99\sin^{2}\alpha\cos^{2}\beta\nonumber \\
	& \qquad-12\sin^{2}\alpha\sin^{2}\beta-18\sin\alpha\cos\alpha\cos\beta\sin\beta)r^{2}\nonumber \\
	& \qquad+16(63\sin^{2}\alpha\cos\alpha\cos\beta-45\sin\alpha\cos\alpha^{2}\sin\beta\nonumber \\
	& \qquad-81\cos^{3}\alpha\cos\beta+3\sin^{3}\alpha\sin\beta)r\nonumber \\
	& \qquad-4(\sin^{4}\alpha+16\sin^{2}\alpha\cos^{2}\alpha+81\cos^{4}\alpha).
	\label{eq:twoclusterbifN3p1}
\end{align}
We summarise the bifurcations for $\beta\neq 0$ in Figure~\ref{fig:N=3betaneq0}. One of the more surprising observations of the analysis above can be summarised as follows.

\begin{cor}
In order for the Hopf bifurcation of splay phase and the bifurcation of in-phase to occur at different points in parameter space for (\ref{eq:oscN}) with $N=3$, it is necessary to have a coupling function $g$ with at least three harmonics.
\end{cor}

\subsection{Bifurcations for $N=4$}

Figure~\ref{fig:N=4betaneq0} illustrates a selection of the local bifurcations that determine the structure of the phase space for $N=4$ in the case $q=-1$ on varying $\alpha\in[0,\pi)$ and $r\geq 0$ for four values of $\beta$. Note that the case $\beta=0$ in the dot-dashed region corresponds to the case considered in \cite{ABM08} for $N=4$.
As before, the in-phase solution has a bifurcation at locations where (\ref{eq:syncbif}) is satisfied. For the splay phase solution with $N=4$ the stability transverse to the rotating block $(0,\pi,\theta,\theta+\pi)$ is governed by the complex eigenvalues
\begin{equation}
\lambda^{(1)}_{\mathrm{splay},N=4}= -\frac{1}{2} q\cos\alpha -\frac{i}{2} q \sin \alpha
\label{eq:N4splay1}
\end{equation}
meaning there is a Hopf bifurcation of the splay phase when
$\cos \alpha=0$, namely at $\alpha=\pi/2$ independent of $\beta$. There can also be bifurcation from splay phase to rotating block where 
\begin{equation}
\lambda^{(2)}_{\mathrm{splay},N=4}= - 2 r \cos \beta,
\label{eq:N4splay2}
\end{equation}
meaning there is a bifurcation from the splay phase to rotating blocks $(0,\pi,\theta,\theta+\pi)$ when $\cos \beta=0$.

\begin{figure}%
	\centerline{
		\includegraphics[width=15cm]{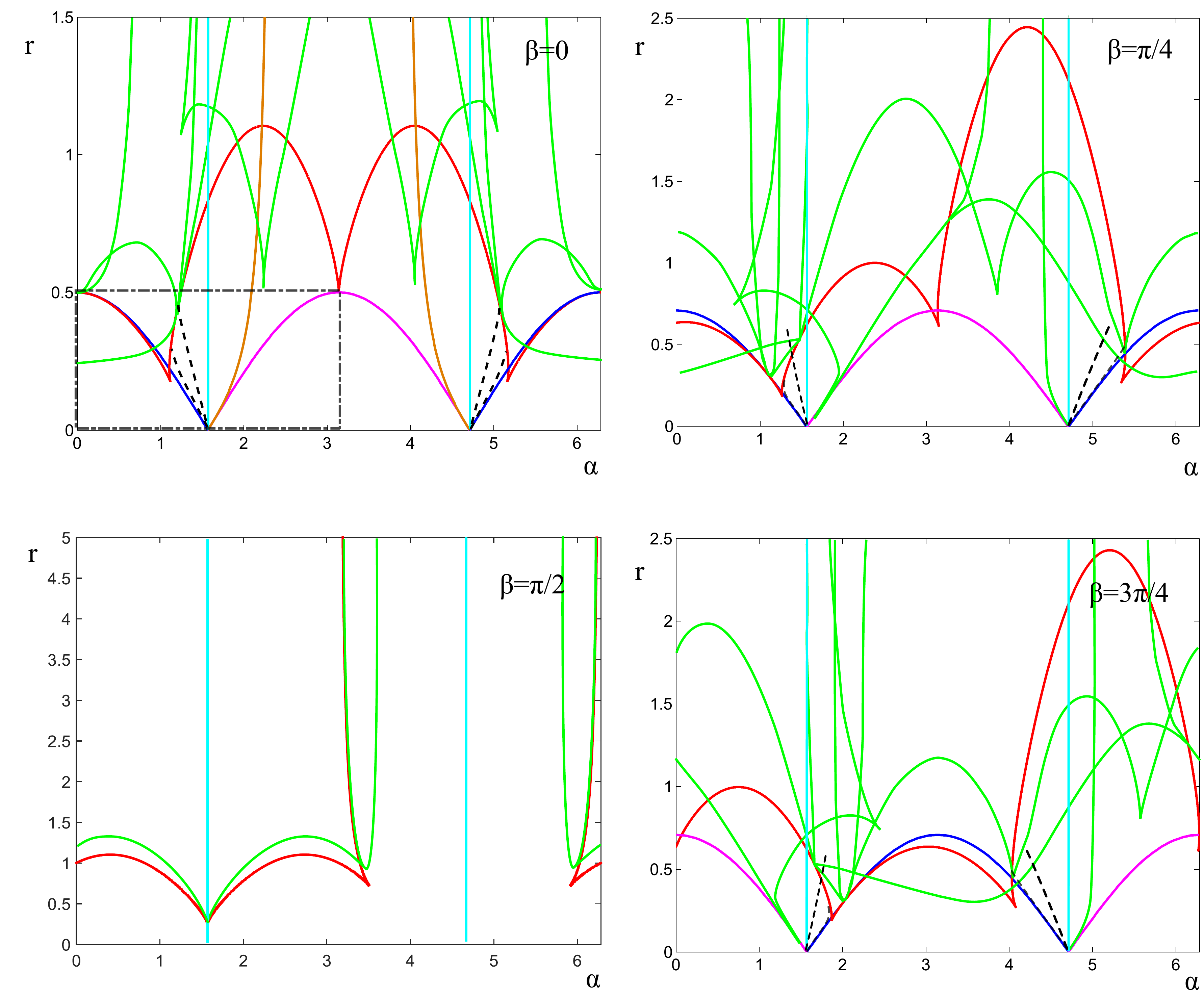}}%
	\caption{Summary of bifurcations for $N=4$ on varying $\alpha$ and $r$ for four choices of $\beta$. The blue lines indicate steady bifurcation of the in-phase (\ref{eq:syncbif}) while the cyan lines indicate Hopf bifurcation of splay phase solutions ($\cos\alpha=0$). The red lines indicate saddle-node bifurcations of two cluster states with isotropy $\Ss_3$. The green lines indicate the location of steady bifurcations involving three cluster states with isotropy $\Ss_2$. The dashed lines indicate (incomplete lines of) bifurcations of periodic orbits. The dot-dashed regions for the case $\beta=0$ corresponds to the diagram in \cite{ABM08} - see there for more details of the dynamics in this case. Note that the case $\beta=\pi/2$ and $\alpha=\pi/2$ (or $3\pi/2$) correspond to an even coupling function.}%
	\label{fig:N=4betaneq0}%
\end{figure}

The dynamics of the points with isotropy $(\Ss_1\times\Ss_1)\times_s \Z_2$ (rotating blocks $(0,\pi,\theta,\theta+\pi)$) is determined by
\[
\dot{\theta}=4r\cos\theta\sin(\theta-\beta)
\]
meaning this is degenerate for $\beta=\pi/2$, independent of $\alpha$! Note that the bifurcations for the rotating blocks only becomes non-degenerate on addition of the fourth harmonic - only then can we get nontrivial solutions with this symmetry. 
There are two possible two-cluster state: for $N=4$ and $p=1$ so that $P=1/4$ and setting $q=-1$, (\ref{eq:twocluster1}) can be expressed as the cubic equation
\begin{equation}
-\sin(\alpha) {t}^{3}-4\,\cos(\beta) r{t}^{2}-2\,\cos(\alpha) {t}^{2}+4\,\sin (\beta) rt-\sin(\alpha) t+4\,\cos(\beta) 
r-2\,\cos(\alpha)=0.
 \label{eq:twoclustN4p1}
\end{equation}
There are bifurcations of two cluster states with isotropy~$\Ss_3$ when the discriminant of~\eqref{eq:twoclustN4p1} (not shown here) is zero.

For isotropy $\Ss_2\times\Ss_2$ we can set $N=4$, $p=2$ and $q=-1$ in (\ref{eq:twocluster1}) to find a quadratic expression
\begin{equation}
-2\,\cos \left( \beta \right) r{t}^{2}-\cos \left( \alpha \right) {t}^
{2}+2\,\cos \left( \beta \right) r-\cos \left( \alpha \right) 
=0
\end{equation}
with discriminant
\begin{equation}
16\, \left( \cos \left( \beta \right)  \right) ^{2}{r}^{2}-4\, \left( 
\cos \left( \alpha \right)  \right) ^{2}
=0.
\end{equation}
Hence the bifurcations of these states are at
\begin{equation}
r = \pm \frac{\cos \alpha}{2\cos \beta}.
\end{equation}
The remainder of the curves on Figure~\ref{fig:N=4betaneq0} are obtained by numerical path following of local bifurcations on the two dimensional fixed point subspace of $\Ss_2$ in the reduced system (\ref{eq:oscNreduced}) using XPPAUT \cite{xppaut}. We remark there may well be more local bifurcations that involve fully symmetry broken (four-cluster) states, as well as global bifurcations that are not shown on this diagram.

\section{Dynamics for odd and even coupling functions}
\label{sec:OddEven}

In the case that the coupling function~$g$ in~\eqref{eq:oscN} is even or odd, the system has additional structure. Write $g = g_+ + g_-$ where 
\begin{align}\label{eq:gpgm}
g_+(\varphi) &=\frac{1}{2}\left(g(\varphi)+g(-\varphi)\right), &
g_-(\varphi) &=\frac{1}{2}\left(g(\varphi)-g(-\varphi)\right) .
\end{align}
are the even and odd parts respectively. For the coupling functions of the form~\eqref{eq:g2} we have
\begin{subequations}
\begin{align}
g_+(\varphi) &= -q\cos(\varphi) \sin(\alpha) - r \cos(2\varphi)\sin(\beta),\\
g_-(\varphi) &= q\sin(\varphi)\cos(\alpha) + r \sin(2\varphi)\cos(\beta).
\end{align}
\end{subequations}
The remaining results in this section are independent of $N$ and apply to general even or odd~$g$.

If $g$ is odd (i.e. $g_+(\varphi)\equiv 0$) then the system for phase differences is a gradient system~\cite{Hoppensteadt1997, Brown2003}. More precisely, if we write $G'(\varphi)=g(\varphi)$ which is single-valued if $g$ is odd and periodic, we can define a potential 
$$
V(\theta)
=-\frac{1}{2N}\sum_{k,j=1}^{N}G(\theta_{k}-\theta_{j})
$$
then (\ref{eq:oscN}) can be written in the form
\[\frac{d}{dt} \theta_k = \omega - \frac{\partial}{\partial \theta_k} V(\theta).\]
This can be seen as a gradient system for the phase differences (\ref{eq:oscNreduced}) on identifying $\T^{N-1}$ as the set $\theta\in\T^{N}$ such that $\sum_{k}\theta_k=0$.

Hence, the only have attractors that are local minima of the potential $V$. In general, for $\omega\neq 0$ one can write a potential for the full system on $\R^N$ but note that this is multivalued on $\T^N$.

The case where the coupling function $g$ is even ($g_-\equiv 0$) is dynamically more interesting: there are hints that it organizes the emergence of chaotic dynamics for nearby parameter values~\cite{Bick2011}.
\begin{lem}
If $g$ is even then the flow of (\ref{eq:oscN}) is divergence free. Moreover, the reduced system (\ref{eq:oscNreduced}) is also divergence free.
\end{lem}

\begin{proof}
To see this, consider
\[
\frac{\partial F_k}{\partial\theta_k} = \frac{1}{N} \sum_{j\neq k} g'(\theta_k-\theta_j)
\]
and so
\[
\sum_{k=1}^N \frac{\partial F_k}{\partial \theta_k} = \sum_{k=1}^N \sum_{j\neq k} g'(\theta_k-\theta_j).
\]
Since $g$ is even, this implies that $g'$ is odd. Thus, $g'(\theta_k-\theta_j)=-g'(\theta_j-\theta_k)$ and we have
\[
\sum_{k=1}^N \frac{\partial F_k}{\partial \theta_k} = 0.
\]
More than this, note that the Jacobian in the direction of the $\T$ action is trivial, meaning that the dynamics of the phase differences (\ref{eq:oscNreduced}) transverse to this direction is also divergence free.
\end{proof}

Recall that an involution $\Rev:\T^N\to\T^N$ is a time-reversal symmetry if
\[\frac{d}{dt}\Rev(\theta) = -F(\Rev(\theta)),\]
that is, $\Rev$ maps any solution to another solution with the direction of the flow is inverted~\cite{Lamb1997}.

\begin{lem}
If $g$ is even then the system~\eqref{eq:oscN} has a time and parameter-reversal symmetry $\Rev:\T^N\to\T^N$ where
\[
\Rev(\theta)=-\theta,~~\omega\mapsto -\omega,~~t\mapsto -t.
\]
For the reduced system (\ref{eq:oscNreduced}) this corresponds to a time-reversal symmetry $\Revt:\T^{N-1}\to\T^{N-1}$ where
\[\Revt(\theta)=-\theta,~~t\mapsto -t.\]
\end{lem}

\begin{proof} 
Note that for $\omega = 0$ we have
\[F_{k}(\theta_{1}, \dotsc, \theta_{N})=\frac{1}{N}\sum_{j=1}^{N}g(\theta_{k}-\theta_{j})=\frac{1}{N}\sum_{j=1}^{N}g(-\theta_{k}+\theta_{j})=F_{k}(-\theta_{1},\dotsc,-\theta_{N}).
\]
At the same time,
\[\frac{d}{dt}\Rev(\theta) = -\frac{d\theta}{dt}\]
which implies that $\Rev$ is a time-reversal symmetry of~\eqref{eq:oscN} for the special case $\omega=0$. This immediately implies that there is a time-reversal symmetry of (\ref{eq:oscNreduced}) in general, but only a time and parameter reversal symmetry in general for (\ref{eq:oscN}).
\end{proof}

Note that given any symmetry $g\in\Gamma$ and a time-reversal symmetry $\Rev$, the composition of the two $\Rev\circ g$ is also a time-reversal symmetry. In fact, the set of time-reversal symmetries on $\T^{N-1}$ is in one-to-one correspondence with the set of symmetries $\Gamma$~\cite{Lamb1997} and can be generated in this way: given any time-reversal symmetry $\Rev$, the set of time-reversal symmetries is $\Rev \Gamma$. A consequence is that if $\theta\in\T^N$ has isotropy $H\subset\Ss_N\times\T$ then so does $\Rev(\theta)$.

If~$\Rev$ is a time reversal-symmetry, the fixed point subspace $\Fix(\Rev)=\set{\theta\in\T^N}{\Rev\theta=\theta}$ is not necessarily dynamically invariant, but it is of interest as it yields a natural way to find families of periodic solutions and homoclinic and heteroclinic orbits~\cite{Lamb1997}. In particular, the reversible Lyapunov Centre Theorem asserts that, subject to nonresonance conditions, near an equilibrium in $\Fix(\Rev)$ with pairs of purely imaginary eigenvalues there is a two parameter family of reversible periodic orbits, even if there are additional zero eigenvalues~\cite{Devaney1976, Golubitsky1991}.

There is a particular time-reversal symmetry $\Revh:\T^{N-1}\to\T^{N-1}$ that maps the canonical invariant region $\Cc$ to itself. More precisely, let us define
\begin{align}
\Revh(0, \theta_2, \dotsc, \theta_N) &=(0, 2\pi-\theta_{N}, \dotsc, 2\pi-\theta_{2}).
\end{align}
on the canonical invariant region and note that this is a reversing symmetry corresponding to reversing the order of the components and composing with $\Rev$. We characterize all points in~$\Cc$ that are mapped to themselves by a reversing symmetry. Thus, we define 
\begin{equation}\label{eq:FPs}
\Rev\Cc := \tset{\theta\in\Cc}{\tau^q\theta = \Revh\theta \text{ for some } q=0,\dotsc,N-1}
\end{equation}
These are the fixed points of time-reversal symmetries within $\Cc$. Let us define 
\begin{equation}
Q^{N, q} := \tset{\theta\in\Cc}{\Revh\circ \tau^{-q} \theta = \theta}
\end{equation} 
and note that
\begin{equation}
\Rev\Cc = \bigcup_{q=0}^{N-1}Q^{N, q}.
\end{equation}

On $\partial\Cc$ we have for $\theta\in\sset{0, \pi}^N\cap\overline{\Cc}\subset \Rev\Cc$. Moreover, the splay state~\eqref{eq:Splay} is contained in $\Rev\Cc$. We now turn to the time-reversal symmetries for $N=3$ and $N=4$ oscillators. In each case we rescale time by~$N$ for simplicity.

\subsection{Dynamics for even coupling functions: $N=3$}

A system for $N=3$ oscillators has and additional degeneracy: there is a constant of motion
\begin{equation}\label{eq:ConstMotRing}
V(\theta) = G(\theta_1-\theta_{2})+G(\theta_2-\theta_{3})+G(\theta_3-\theta_{1})
\end{equation}
with $G'=g$ which generalizes a specific case of the Watanabe--Strogatz constant of motion~\cite{Bick2012b, Pikovsky2006}. Hence, the dynamics of the system is effectively one-dimensional.

\begin{figure}
	\begin{center}
		\subfigure[$\theta\in Q^{3, 0}$]{\includegraphics[scale=0.6]{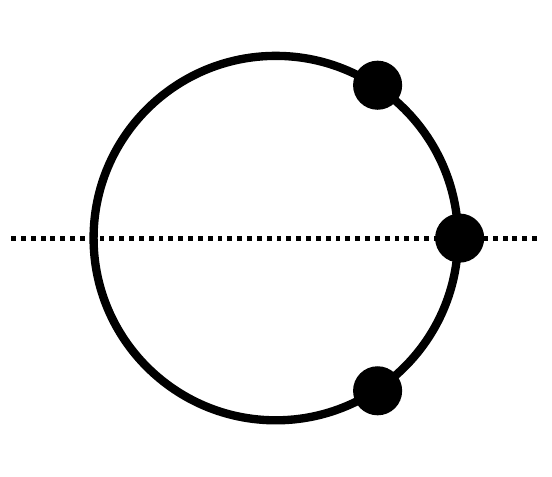}}\qquad\
		\subfigure[$\theta\in Q^{3, 1}$]{\includegraphics[scale=0.6]{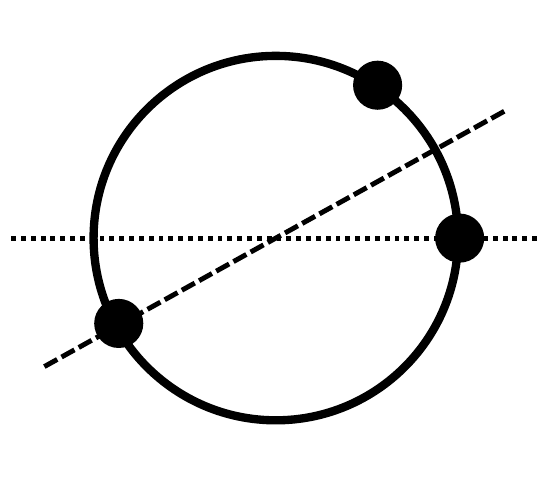}}\qquad\
		\subfigure[$\theta\in Q^{3, 2}$]{\includegraphics[scale=0.6]{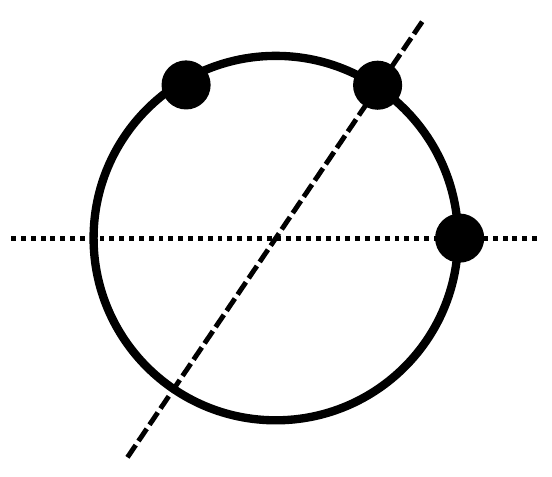}}
		\end{center}
	\caption{\label{fig:Q3}Configurations of phases of three oscillators that are invariant under reflections. Phases $0=\theta_1< \theta_2< \theta_3$ are aligned anticlockwise around the circle and the $\Z_3$ symmetry operates by rotating the operators anticlockwise by a common phase shift until $\theta_3=0$ (and then relabelling oscillators). For $Q^{3,q}$, reflection along the dashed line corresponds to $\theta\mapsto\Revh\tau^{-q}(\theta)$.
	}
\end{figure}

For the relative fixed point sets of the time-reversal symmetry, we have
\begin{subequations}
\begin{align}
	Q^{3, 0} &= \set{(0, \theta_2, 2\pi-\theta_2)}{\theta_2 \in (0, \pi)}\\
	Q^{3, 1} &= \set{(0, \theta_2, \pi + \theta_2/2)}{\theta_2 \in (0, \pi)}\\
	Q^{3, 2} &= \set{(0, \theta_2, 2\theta_2)}{\theta_2 \in (0, \pi)}
\end{align}
\end{subequations}
which define three line segments that meet at the splay state. 
The configurations of the oscillators are depicted in Figure~\ref{fig:Q3}. Figure~\ref{fig:FixRN3} shows $\Rev\Cc = Q^{N,0} \cup Q^{N,1} \cup Q^{N,2}$ superimposed on the dynamically invariant curves determined by the level sets of~\eqref{eq:ConstMotRing} for an even coupling function with four Fourier modes:
\begin{equation}
g(\varphi)=\sum_{j=0}^{4} c_j \cos(j\varphi).
\label{eq:geven4}
\end{equation}
The points with time-reversal symmetry $\Rev\Cc$ intersect the boundary of the canonical invariant region at $(0, 0, 0)$, $(0, 0, \pi)$, $(0, \pi, \pi)$.

\begin{figure}
	\begin{center}
		\subfiglabelskip=0pt
		\subfigure[]{\includegraphics[scale=1.1]{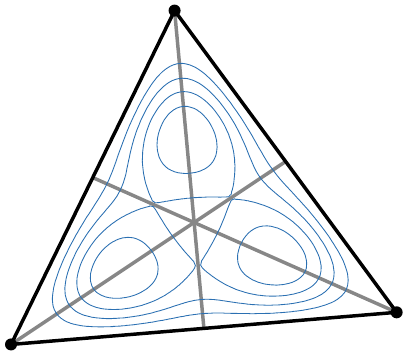}}\qquad\
		\subfigure[]{\includegraphics[scale=1.1]{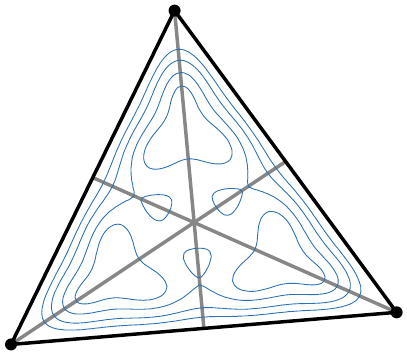}}\qquad\
		\subfigure[]{\includegraphics[scale=1.1]{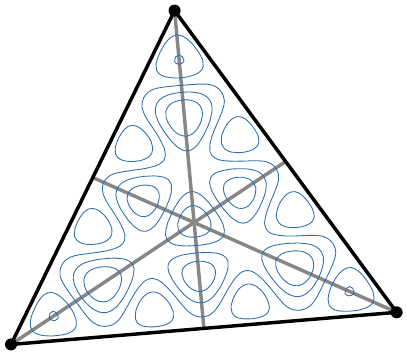}}
	\end{center}
	\caption{\label{fig:FixRN3}
For $N=3$ oscillators, the points of $\Rev\Cc$ organize the dynamics in~$\Cc$ for even coupling functions. As in Figure~\ref{fig:CIR}(b), the boundary of~$\Cc$ are points of~$\Ss_2$ (black lines) which intersect in~$\Sync$ (black dot). The union of fixed point subspaces~$\Rev\Cc$ for time-reversal symmetries are depicted by gray lines which intersect in~$\Splay$. The level curves (blue lines) of the constant of motion~\eqref{eq:ConstMotRing} for different
parameter values of the coupling function~\eqref{eq:geven4} are superimposed in each panel: the nontrivial Fourier coefficients are $c_1 = 1, c_2 = 1$ in Panel~(a), $c_1 = -2$, $c_2 = -2$, $c_3 = -1$, and $c_4 = -0.88$ in Panel~(b), and $c_1 = -2$, $c_2 = -2$, $c_3 = -1$, and $c_4 = 10$ in Panel~(c). Note that in the third panel we are close to an additional symmetry: a high order Fourier component dominates.
	}
\end{figure}

To find the equilibria in $\Rev\Cc$ it suffices to evaluate the equilibria in $Q^{3,0}$ since $Q^{3,1}$, $Q^{3,2}$ are the images of $Q^{3, 0}$ under the action of~$\Z_3$. Write $\psi_k = \theta_k-\theta_1$ and $F^\psi_k = F_k-F_1$. On~$Q^{3, 0}$ we have the vector field
\begin{subequations}
\begin{align}
F^\psi_2(\psi_2, 2\pi-\psi_2) &= g(2\psi_2)-g(\psi_2)\\
F^\psi_3(\psi_2, 2\pi-\psi_2) &= g(2\psi_2)-g(\psi_2)
\end{align}
\end{subequations}
perpendicular to $Q^{3, 0}$. A point $(\psi_2^\star, \psi_3^\star)$ is an equilibrium if $g(2\psi_2^\star)-g(\psi_2^\star)=0$ and $\psi_3^\star = 2\pi-\psi_2^\star$.
The linearization of~$F^\psi$ at $(\psi_2^\star, \psi_3^\star)$,
\begin{equation}
	DF^\psi(\psi_2^\star, \psi_3^\star)=\left(\begin{array}{cc}    
		g'(\psi_2^\star-\psi_3^\star)& -g'(\psi_2^\star-\psi_3^\star)-g'(\psi_3^\star)\\
		g'(\psi_2^\star-\psi_3^\star)-g'(\psi_2^\star)& -g'(\psi_2^\star-\psi_3^\star)\end{array}\right),
\end{equation}
has eigenvalues
\begin{align*}
\lambda &= \pm\sqrt{g'(\psi_2^\star-\psi_3^\star)g'(\psi_2^\star) - g'(\psi_2^\star-\psi_3^\star)g'(\psi_3^\star)+g'(\psi_2^\star)g'(\psi_3^\star)}\\
&=\pm\left({\sum_{j=1}^3 g'(\theta_j^\star-\theta_{j-1}^\star)g'(\theta_j^\star-\theta_{j+1}^\star)}\right)^{\frac{1}{2}}
\end{align*}
where indices are taken modulo $N=3$. The eigenvalues thus always take either real or pure imaginary values.

Figure~\ref{fig:FixRN3} shows examples where it appears that the only equilibria in~$\Cc$ lie in $\Rev\Cc$. Equilibria with a pair of imaginary eigenvalues give rise to ``islands'' of neutrally stable periodic orbits. These are organized by the stable and unstable manifolds of the saddle points on~$\Rev\Cc$ which form heteroclinic connections and yield the topological obstructions bounding the islands. The equilibria should relate directly to the constant of motion where the level sets are not smooth manifolds.

\subsection{Dynamics for even coupling functions: $N=4$}
\label{ex:N4}

For $N=4$ we calculate
\begin{subequations}
\begin{align}
	Q^{4,0} &= \set{(0, \theta_2, \pi, 2\pi-\theta_2)}{\theta_2 \in (0, \pi)}\\
	Q^{4,1} &= \set{(0, \theta_2, \theta_3, 2\pi -\theta_3+ \theta_2)}{0<\theta_2<\theta_3<2\pi -\theta_3+ \theta_2<2\pi}\\
	Q^{4,2} &= \set{(0, \theta_2, 2\theta_2, \pi+\theta_2)}{\theta_2 \in (0, \pi)}\\
	Q^{4,3} &= \set{(0, \theta_2, \theta_3, \theta_2+\theta_3)}{0<\theta_2<\theta_3<\theta_2+\theta_3<2\pi}.
\end{align}
\end{subequations}
The configuration of phases are depicted in Figure~\ref{fig:Q4}. Note that $\tau Q^{4,0} = Q^{4,2}$ and $\tau Q^{4,1} = Q^{4,3}$.
The sets $Q^{4,1}, Q^{4,3}$ are subsets of planes that intersect the edges of the canonical invariant in the cluster states with $\Ss_2\times\Ss_2$ isotropy and lines on the faces corresponding to cluster states with $\Ss_2$ isotropy. The planes intersect in the interior of the canonical invariant region in the set of points with~$\Z_2$ isotropy\footnote{Note that for $N>4$ there are points with $\Z_k\subset\Z_N$ isotropy that are not in $\Rev\Cc$.} and subdivide the CIR into four connected components. The sets $Q^{4,0}, Q^{4,2}$ are line segments that intersect~$\partial\Cc$ in $(0, \pi, \pi, \pi)$, $(0, 0, \pi, 0)$, $(0, 0, 0, \pi)$, and $(0, \pi, 0, 0)$. We have
$\bigcap_{j=0}^{3}Q^{4,j} = \Splay$. Figure~\ref{fig:N4Space} shows a projection of the canonical invariant region and~$\Rev\Cc$ for coupling function~\eqref{eq:geven4} with four nontrivial Fourier modes.
Note that while $\Rev\Cc$ is not necessarily invariant as fixed point subspaces of a time-reversal symmetry, it may contain invariant subsets, such as points of nontrivial isotropy.

\begin{figure}
	\begin{center}
		\subfigure[$\theta\in Q^{4,0}$]{\includegraphics[scale=0.6]{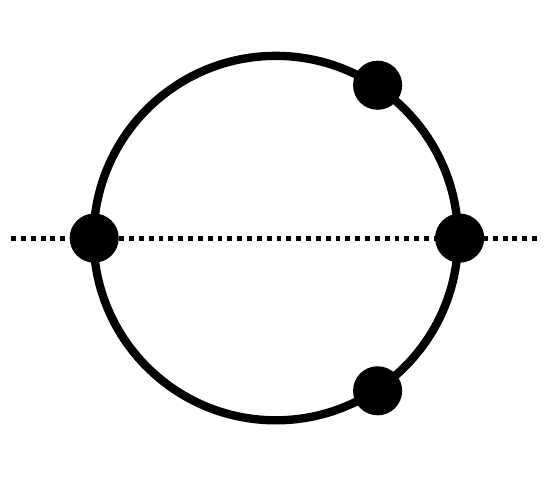}}
		\quad
		\subfigure[$\theta\in Q^{4,1}$]{\includegraphics[scale=0.6]{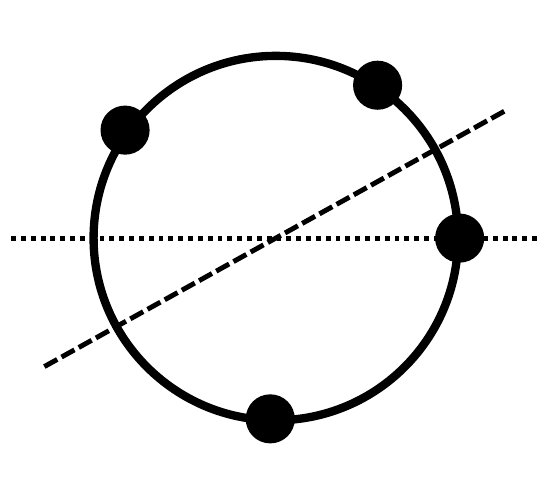}}\quad
		\subfigure[$\theta\in Q^{4,2}$]{\includegraphics[scale=0.6]{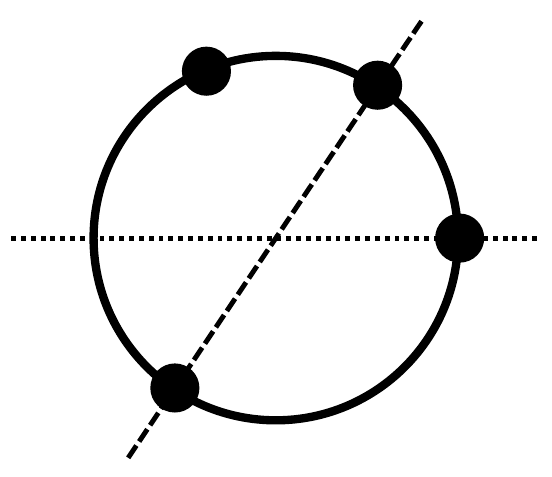}}
		\quad
		\subfigure[$\theta\in Q^{4,3}$]{\includegraphics[scale=0.6]{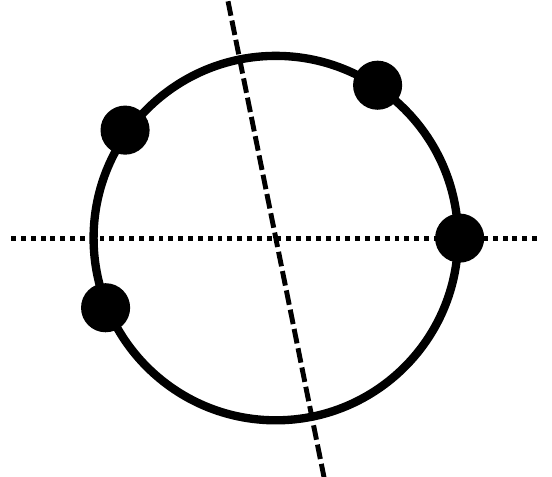}}
	\end{center}
	\caption{\label{fig:Q4}Configurations of phases of four oscillators that are invariant under time-reversal. Phases $0=\theta_1< \theta_2< \theta_3<\theta_4$ are aligned anticlockwise around the circle. Reflection along the dashed line for $Q^{N,q}$ corresponds to the action of time-reversal symmetry $\theta\mapsto\Revh\tau^{-q}(\theta)$.
	}
\end{figure}

\begin{figure}
	\begin{center}
		\includegraphics[scale=1.1]{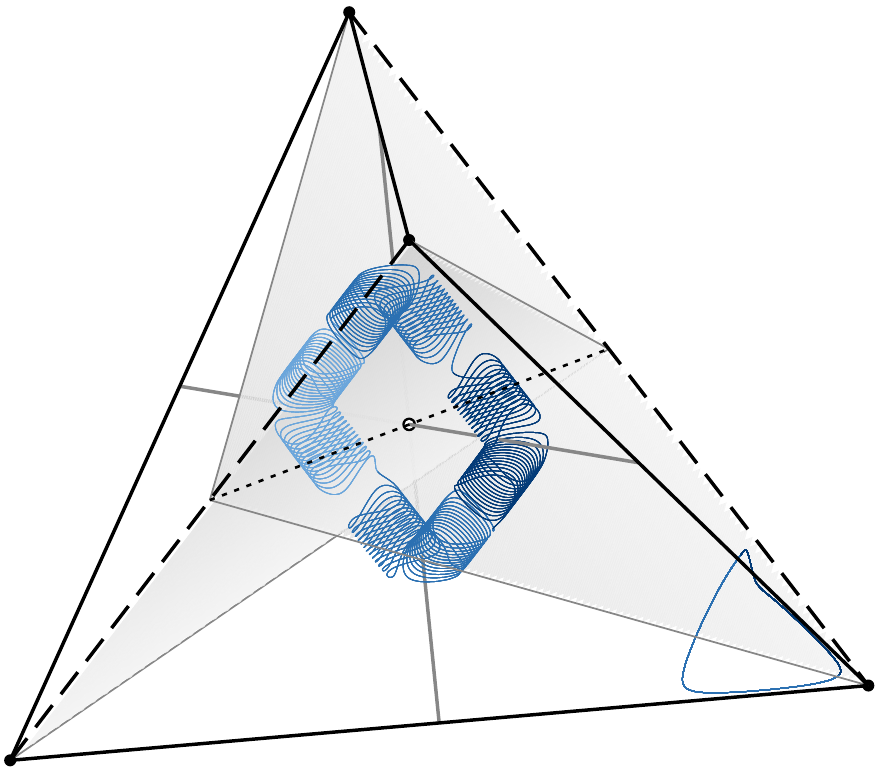}
	\end{center}
	\caption{\label{fig:N4Space}
	For $N=4$ oscillators, the points of $\Rev\Cc$ constrain the dynamics in~$\Cc$ for even coupling functions. As in Figure~\ref{fig:CIR}(d), the faces of~$\Cc$ are points of~$\Ss_2$ isotropy and black lines correspond to points of isotropy $\Ss_3$ (solid), $\Ss_2\times\Ss_2$ (dashed), and $\Z_2$ (dotted). The points in $\Rev\Cc$ are depicted in gray: $Q^{4,0}, Q^{4,2}$ (thick gray lines), $Q^{4,1}, Q^{4,3}$ (planes and thin gray lines that depict points with~$\Ss_2$ isotropy within). Trajectories (blue lines) for two initial conditions are plotted for the coupling function~\eqref{eq:geven4} with coefficients $a_1=-0.5$, $a_2=-0.5$, $a_3=-0.25$, $a_4=10$. The shading of the trajectories indicates in which of the four connected subsets a point lies and changes of colour correspond to intersections of the trajectories with~$Q^{4,1}$ and~$Q^{4,3}$. Note that the trajectories intersect~$Q^{4,1}$ and $Q^{4,3}$ in exactly two points. For this choice of coupling function a higher order Fourier component dominates as in Figure~\ref{fig:FixRN3}(c) leading to the ``spiraling'' motion of the trajectory; details will be discussed elsewhere.
	}
\end{figure}

We give equations for fixed points in $Q^{4,0}$. In fact, any equilibrium within~$\Rev\Cc$ for $N=4$ oscillators has to be neutrally stable after reducing the phase-shift symmetry; the linearization around equilibria with time-reversing symmetry for a three dimensional system will have one pair of real or pure imaginary eigenvalues and one zero eigenvalue. 
Introducing phase difference variables $\psi_k=\theta_k-\theta_1$ for $Q^{4,0}$, we 
have $\psi_3=\pi$, $\psi_4 = 2\pi-\psi_2$. The vector field with components $F_k^\psi = F_k-F_1$ is given by
\begin{subequations}\label{eq:Q40}
\begin{align}
F_2^\psi(\psi_2, \pi, 2\pi-\psi_2) &= g(2\psi_2)+g(\psi_2+\pi)-g(\psi_2)-g(\pi),\\
F_3^\psi(\psi_2, \pi, 2\pi-\psi_2) &= 2(g(\psi_2+\pi)-g(\psi_2)),\\
F_4^\psi(\psi_2, \pi, 2\pi-\psi_2) &= g(2\psi_2)+g(\psi_2+\pi)-g(\psi_2)-g(\pi).
\end{align}
\end{subequations}
Note that $F_2^\psi=F_4^\psi$ along~$Q^{4,0}$. The linearization of~$F^\psi$ at an equilibrium $(0, \psi^\star, \pi, 2\pi-\psi^\star)$ is
\begin{equation}
DF^\psi(\psi^\star, \pi, 2\pi-\psi^\star) = \left(\begin{array}{ccc}
g'(\psi^\star+\pi)+g'(2\psi^\star) & -g'(\psi^\star+\pi) & -g'(2\psi^\star)+g'(\psi^\star)\\
g'(\psi^\star+\pi)-g'(\psi^\star) & 0 & -g'(\psi^\star+\pi)+g'(\psi^\star)\\
g'(2\psi^\star)-g'(\psi^\star) & g'(\psi^\star+\pi) & -g'(2\psi^\star)-g'(\psi^\star+\pi)
\end{array}\right)
\end{equation}
which has eigenvalues
\begin{equation}\label{eq:EvN4Q0}
\lambda \in \sset{0, \pm\sqrt{2(g'(\psi^\star)g'(\psi^\star+\pi)+g'(\psi^\star+\pi)g'(2\psi^\star)+g'(2\psi^\star)g'(\psi^\star))-g'(\psi^\star)^2-g'(\psi^\star+\pi)^2}}.
\end{equation}
These simplify if some harmonics of the coupling function~$g$ vanish: if the coupling function is $\pi$-periodic, $g(\phi+\pi)=g(\phi)$, then~\eqref{eq:EvN4Q0} yields
\[\lambda \in \sset{0, \pm 2\sqrt{g'(\psi^\star)g'(2\psi^\star)}}\]
and if $g(\phi+\pi)=-g(\phi)$ we obtain
\[\lambda \in \sset{0, \pm 2 \imath g'(\psi^\star)}\]
implying that neutrally stable periodic orbits arise nearby.

Similarly, one can calculate equilibria in $Q^{4,3}$ where we have $\psi_4 = \psi_2+\psi_3$.
For these points, the dynamics are given by
\begin{subequations}
\begin{align}
F_2^\psi(\psi_2, \psi_3, \psi_2+\psi_3) &= g(\psi_2-\psi_3)-g(\psi_2+\psi_3),\\
F_3^\psi(\psi_2, \psi_3, \psi_2+\psi_3) &= g(\psi_2-\psi_3)-g(\psi_2+\psi_3),\\
F_4^\psi(\psi_2, \psi_3, \psi_2+\psi_3) &= 0.
\end{align}
\end{subequations}
Note that the vector field vanishes in the $\psi_4$ direction along~$Q^{4,3}$. Points $(0, \psi_2^\star, \psi_3^\star, \psi_2^\star+\psi_3^\star)$ are equilibria if 
\begin{equation}\label{eq:FixedN4}
g(\psi_2^\star-\psi_3^\star)-g(\psi_2^\star+\psi_3^\star)=0.
\end{equation}
Similarly calculating the eigenvalues, we find
\begin{equation}
\lambda \in \sset{0, \pm\sqrt{
2g'(\psi_2^\star-\psi_3^\star)(g'(\psi_2^\star)-g'(\psi_3^\star))+2g'(\psi_2^\star+\psi_3^\star)(g'(\psi_2^\star)+g'(\psi_3^\star))
}}
\end{equation}
and the eigenvector for $\lambda=0$ is given by
\begin{equation}\label{eq:ZeroEV}
v_0 = \left(\begin{array}{c}
g'(\psi_2^\star-\psi_3^\star)+g'(\psi_2^\star+\psi_3^\star)\\
g'(\psi_2^\star-\psi_3^\star)-g'(\psi_2^\star+\psi_3^\star)\\
2g'(\psi_2^\star-\psi_3^\star)
\end{array}\right)
\end{equation}

Since $g$ is even, there are solutions to the fixed point equation~\eqref{eq:FixedN4} that are independent of~$g$. The arguments being equal, $\psi_2^\star-\psi_3^\star = \psi_2^\star+\psi_3^\star+2q\pi$, $q\in\Z$, 
implies $\psi_3^\star = q\pi$, $q\in\Z$ and arbitrary~$\psi_2^\star$. If the sign of the arguments is opposite, we have $\psi_2^\star-\psi_3^\star = -\psi_2^\star-\psi_3^\star+2q\pi$, $q\in\Z$ 
which implies $\psi_2^\star = q\pi$, $q\in\Z$. The condition $\psi_3^\star=\pi$ defines a line of equilibria given by
\[L_- = \set{(0, \varphi, \pi, \pi+\varphi)}{\varphi\in(0,\pi)}\]
which are points of~$\Z_2$ isotropy in~$\Cc$. Note that the eigenvector~$v_0$ of the trivial eigenvalue,~\eqref{eq:ZeroEV}, points along~$L_-$.
Moreover, the remaining two eigenvalues evaluate to
\begin{equation}\label{eq:LmStab}
\lambda = \pm 2\sqrt{g'(\varphi)g'(\varphi+\pi)}.
\end{equation}
If some harmonics of the coupling function~$g$ vanish they take either real or imaginary values for all points in~$L_-$: if the coupling function is $\pi$-periodic, $g(\phi+\pi)=g(\phi)$, $L_-$ is an continuum of saddles since
\[\lambda = \pm 2 g'(\varphi).\]
If $g(\phi+\pi)=-g(\phi)$ we obtain
\[\lambda = \pm 2\imath g'(\varphi)\]
implying that neutrally stable periodic orbits arise close to~$L_-$.

Similarly, in the second case where the sign of the arguments is opposite, consider $\psi_2^\star=0$ which yields a continuum of equilibria 
\[L_+ = \set{(0, 0, \varphi, \varphi)}{\varphi\in(0,2\pi)}\]
which are points of $\Ss_2\times\Ss_2$ isotropy on $\partial\Cc$. 
Again, the eigenvector~$v_0$ corresponding to the neutrally stable direction points along~$L_+$. The remaining eigenvalues are given by
\begin{equation}\label{eq:LpStab}
\lambda = \pm 2g'(\varphi)
\end{equation}
yielding a family of saddles. The stable and unstable manifolds lie in boundary faces of $\Cc$ with $\Ss_2$ isotropy. They can intersect~$\Rev\Cc$ again in $Q^{4,1}$ which intersects the boundary faces adjacent to $L_+$ in a line.

While the equilibria studied above exist for any choice of even coupling function~$g$, there may be further equilibria that organize the dynamics. First, the equations~\eqref{eq:Q40}, \eqref{eq:FixedN4} may have more equilibria that depend on the coupling function chosen. As these lie in~$\Rev\Cc$ they have symmetric stability properties. 
Second, the equilibria that are not contained in~$\Rev\Cc$ appear to play an important role and yield a mechanism that could give rise to complex heteroclinic networks in the system. For example, assume that~$\theta^\star\in\Cc\smallsetminus\Rev\Cc$ is a hyperbolic equilibrium with trivial isotropy. Since the vector field is divergence free, the equilibrium has to be a saddle. Suppose that $\theta^\star$ has one-dimensional stable and two-dimensional unstable manifold. The action of symmetries and time-reversal implies that there are~$N$ symmetric copies of~$\theta^\star$. At the same time, there are~$N$ copies with inverse stability properties that are given by the group orbit of~$\Revh(\theta^\star)$. 
An intersection of the unstable manifold of~$\theta^\star$ with~$Q^{4,1}$ or~$Q^{4,3}$ yields a heteroclinic connection to an equilibrium related by the time-reversal symmetry. Such a heteroclinic connection may have trivial isotropy (given the way $Q^{4,1}$ and~$Q^{4,3}$ subdivide the CIR into connected components) or nontrivial isotropy, depending on where this intersection takes place. In this way, a network between the (symmetrically and reversing symmetry related) copies of $\theta^\star$ may be formed.

\section{Discussion}
\label{sec:Discuss}

\subsection{Even coupling and integrability}

As noted in \eqref{eq:ConstMotRing} for $N=3$ and arbitrary even coupling function, there is $N-2=1$ constant of the motion, in agreement with the case for Kuramoto coupling \cite{Watanabe1994}. It is natural to consider whether there are $N-2$ independent constants of the motion for $N=4$ and arbitrary even coupling function. If this was the case, the only invariant sets would be two frequency quasiperiodicity. Although~\cite{Watanabe1994} and the presence of time-reversal symmetries might hints this might be the case, we present evidence there is at most one integral of the motion for $N=4$.

In particular, for $N=4$ and the states with $\Ss_2$ isotropy on the boundary of $\Cc$, Figure~\ref{fig:N4nonintegrable} demonstrates that for the case $N=4$ and even coupling functions with up to four harmonics of the form~\eqref{eq:geven4} are not fully integrable. In cases (c), (d) of the Figure, we argue there are no nontrivial smooth integrals of the motion within the $\Ss_2$ subspace, owing to the appearance of sink/source pairs. We only find such evidence of nonintegrable behaviour if $c_3$ and $c_4$ are non both zero: if $c_3$ and $c_4$ are both zero \eqref{eq:geven4} (only two harmonics) then there are apparently no equilibria within $\Ss_2$ that are not fixed by a time reversal symmetry and the phase portraits are consistent with the presence of a smooth constant of the motion.

It is still possible that there is one integral of the motion for the case $N=4$ and general even coupling, as long as this integral is zero on the boundary of the canonical invariant region, and the trajectory shown in Figure~\ref{fig:N4Space} is consistent with this.

\begin{figure}
	\begin{center}
	\includegraphics[width=15cm]{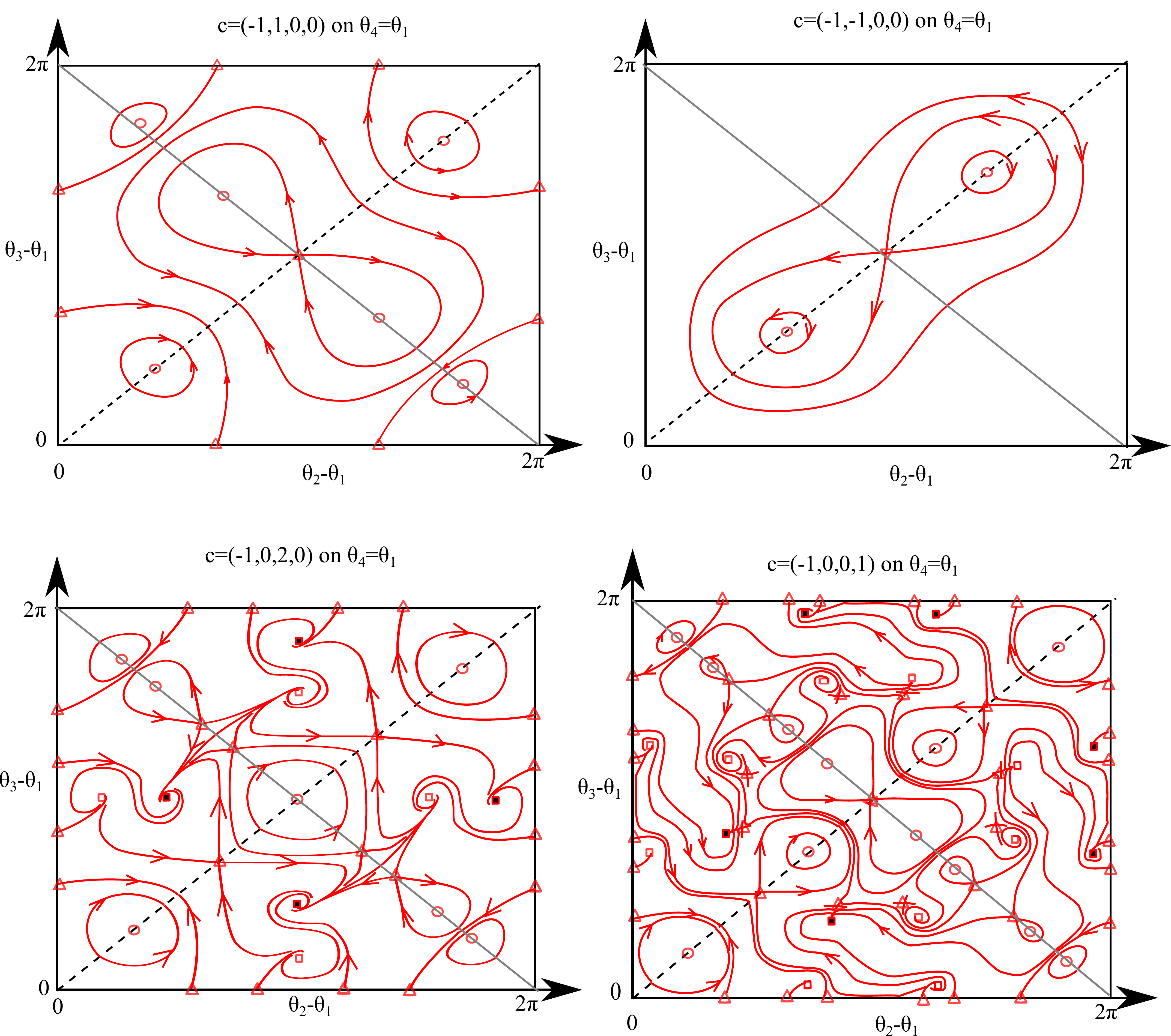}
	\end{center}
	\caption{
	Phase portraits for $N=4$ oscillators show a wide range of continua of heteroclinic cycles on the boundary of~$\Cc$. Here we show a face of the CIR with~$\Ss_2$ isotropy, assuming even~$g$ with the first four harmonics given by~\eqref{eq:geven4}. Dashed black lines are points with $\Ss_2\times\Ss_2$~isotropy foliated with fixed points and gray lines indicate~$\Rev\Cc$. For equilibria,  triangles are saddles, circles are centres, open squares are sources, and filled squares are sinks; these points are degenerate on the dashed line with~$\Ss_2\times\Ss_2$ isotropy which corresponds to~$L_+$ with transversal stability given by~\eqref{eq:LpStab}. Observe that four our choice of~$c_1$, $c_2$, all equilibria lie within~$\Rev\Cc$ if $c_3=c_4=0$, but there are sink/source pairs (indicating lack of integrability) for cases where one of them is non-zero.
	\label{fig:N4nonintegrable} 
	}
\end{figure}

\subsection{Open questions}

We finish by highlighting some open questions raised in this paper. Firstly, we note that addition of the second harmonic does unfold many degeneracies, though not all. We note that there may be some that are still present for (\ref{eq:oscN}) even for coupling functions with an infinite number of harmonics present, and these will disappear on considering more general coupling that is not of ``pairwise'' type (see \cite{AR15}). There seems to be a subtle interplay of the order~$N$, the symmetries of~$g$ and the number of harmonics~$L$ that determine the dynamics.

Extending a bifurcation analysis to higher~$N$ will be harder than extending to higher~$L$, but the latter seems worthwhile especially if it is possible to determine the location of global bifurcations.  We do extend the bifurcation analysis here in part analytically, using the special form of the coupling function to write the bifurcations of two-cluster in polynomial form. Potentially this can be extended to more general cluster states, though at the expense of ever more complex expressions. As an alternative approach, rather than doing a bifurcation analysis for fixed~$L$, one could conversely ask whether some harmonics are crucial for specific bifurcations to occur.

The dynamics for the case of an even coupling function shows remarkable subtlety. Numerical experiments show a variety of periodic orbits which can be quite complicated in examples of even coupling functions where a higher harmonic dominates; cf.~Figure~\ref{fig:N4Space}. Although we demonstrate in the previous section that one cannot expect there to be fully integrability with~$N-2$ independent constants of the motion~\cite[Theorem, p249]{Watanabe1994}, there may well be one constant of the motion in the general even coupling case which we have not yet been able to determine.

For even coupling we demonstrate for $N=3,4$ that the set of phases with time-reversal symmetry within the canonical invariant region is a nontrivial (typically noninvariant) set that organizes the dynamics. Complicated dynamics, such as chimera states~\cite{Abrams2004} and chaotic dynamics~\cite{Bick2011} (and as a consequence also the chaotic weak chimeras constructed in~\cite{Bick2016}), typically arise in systems for small perturbations away from a system with time-reversal symmetry. We anticipate that understanding these heteroclinic networks in greater detail will shed light on how complicated arise in networks of coupled phase oscillators. Naturally, the more equilibria there are, the more intricate the heteroclinic network can be. Thus, deriving conditions for the existence of equilibria depending on the coupling function seems like a first step to understand the dynamics for even coupling functions~$g$.

\subsection*{Acknowledgements}

PA gratefully acknowledges the financial support of the EPSRC via grant EP/N014391/1. CB gratefully acknowledges financial support from the People Programme (Marie Curie Actions) of the European Union's Seventh Framework Programme (FP7/2007--2013) under REA grant agreement no.~626111. 

\bibliographystyle{plain}

\begin{thebibliography}{10}

\bibitem{Abrams2004}
D.~M. Abrams and S.~H. Strogatz.
\newblock {Chimera States for Coupled Oscillators}.
\newblock {\em Physical Review Letters}, 93(17):174102, 2004.

\bibitem{Acebron2005}
J. Acebr\'{o}n, L.~Bonilla, C. {P\'{e}rez Vicente}, F. Ritort, and  R. Spigler.
\newblock {The Kuramoto model: A simple paradigm for synchronization
  phenomena}.
\newblock {\em Reviews of Modern Physics}, 77(1):137--185, 2005.

\bibitem{ABM08}
P.~Ashwin, O.~Burylko and Y.~Maistrenko.
\newblock Bifurcation to heteroclinic cycles and sensitivity in three and four
  coupled phase oscillators.
\newblock {\em Physica D}, 237:454--466, 2008.

\bibitem{Ashwin2014a}
P. Ashwin and O. Burylko.
\newblock {Weak chimeras in minimal networks of coupled phase oscillators}.
\newblock {\em Chaos}, 25:013106, 2015.

\bibitem{ACN15}
P. Ashwin, S. Coombes and R. Nicks.
\newblock {Mathematical frameworks for oscillatory network dynamics in
  neuroscience}.
\newblock {\em Journal of Mathematical Neuroscience} , 6:1--92, 2016.

\bibitem{AOWT07}
P.~Ashwin, G.~Orosz, J.~Wordsworth and S.~Townley.
\newblock Dynamics on networks of clustered states for globally coupled phase
  oscillators.
\newblock {\em SIAM J. Appl. Dynamical Systems}, 6(4):728--758, 2007.

\bibitem{AR15}
P. Ashwin and A. Rodrigues.
\newblock {Hopf normal form with $\Ss_N$ symmetry and reduction to systems of nonlinearly coupled phase oscillators}.
\newblock {\em Physica D}, 325:14--24 2016.

\bibitem{AS92}
P. Ashwin and J.~W. Swift.
\newblock {The dynamics of $n$ weakly coupled identical oscillators}.
\newblock {\em Journal of Nonlinear Science}, 2(1):69--108, 1992.

\bibitem{Bick2012b}
C. Bick.
\newblock {\em {Chaos and Chaos Control in Network Dynamical Systems}}.
\newblock PhD dissertation, Georg-August-Universit{\"{a}}t G{\"{o}}ttingen, 2012.

\bibitem{Bick2016}
C. Bick and P. Ashwin.
\newblock {Chaotic Weak Chimeras and their Persistence in Coupled Populations of Phase Oscillators}.
\newblock {\em Nonlinearity}, 29(5):1468--1486, 2016.

\bibitem{Bick2011}
C. Bick, M. Timme, D. Paulikat, D. Rathlev and P. Ashwin.
\newblock {Chaos in Symmetric Phase Oscillator Networks}.
\newblock {\em Physical Review Letters}, 107(24):244101, 2011.

\bibitem{Brown2003}
E. Brown, P. Holmes and J. Moehlis.
\newblock {Globally coupled oscillator networks}.
\newblock In {\em Perspectives and Problems in Nonlinear Science: A Celebratory
	Volume in Honor of Larry Sirovich}, pages 183--215. Springer, 2003.

\bibitem{BP08}
O. Burylko and A. Pikovsky.
\newblock {Desynchronization transitions in nonlinearly coupled phase oscillators}.
\newblock {\em Physica D}, 240:1352--1361, 2011.

\bibitem{Daido}
H.~Daido.
\newblock Onset of cooperative entrainment in limit-cycle oscillators with
uniform all-to-all interactions: bifurcation of the order function.
\newblock {\em Physica D}, 91:24--66, 1996.

\bibitem{Devaney1976}
R.~L.~Devaney.
\newblock Reversible Diffeomorphisms and Flows.
\newblock {\em Transactions of the American Mathematical Society}, 218:89--113, 1976.

\bibitem{xppaut}
G.~B. Ermentrout.
\newblock {\em A Guide to XPPAUT for Researchers and Students}.
\newblock SIAM, Pittsburgh, 2002.

\bibitem{GolSte} 
M.~Golubitsky, D.G.~Schaeffer and I.N.~Stewart.
\newblock {\em Singularities and Groups in Bifurcation Theory, Vol. II} 
\newblock Appl. Math. Sci. {\bf 69}, Springer-Verlag, New York, 1988.

\bibitem{Golubitsky1991} 
M.~Golubitsky, M.~Krupa and C.~Lim.
\newblock Time-Reversibility and Particle Sedimentation. 
\newblock {\em SIAM Journal on Applied Mathematics}, 51:49--72, 1991.

\bibitem{HMM93}
D.~Hansel, G.~Mato and C.~Meunier.
\newblock Clustering and slow switching in globally coupled phase oscillators.
\newblock {\em Phys. Rev. E}, 48(5):3470--3477, 1993.

\bibitem{Hoppensteadt1997}
F.~C. Hoppensteadt and E.~M. Izhikevich.
\newblock {\em {Weakly Connected Neural Networks}}, volume 126 of {\em Applied
	Mathematical Sciences}.
\newblock Springer, New York, NY, 1997.

\bibitem{K15}
H.~Kori, Y.~Kuramoto, S.~Jain, I.~Z. Kiss and J~Hudson.
\newblock Clustering in globally coupled oscillators near a {H}opf bifurcation:
Theory and experiments.
\newblock {\em Phys. Rev. E}, 89:062906, 2014.

\bibitem{Kuramoto1975}
Y.~Kuramoto.
\newblock Self-entrainment of a population of coupled non-linear oscillators.
\newblock In H.~Araki, editor, {\em International Symposium on Mathematical
	Problems in Theoretical Physics}, Lecture Notes in Physics, pages 420--422.
Springer-Verlag Berlin, 1975.

\bibitem{kuramoto}
Y.~Kuramoto.
\newblock {\em Chemical Oscillations, Waves and Turbulence}.
\newblock Springer-Verlag, Berlin, 1984.

\bibitem{Laing2009}
C.~R. Laing.
\newblock The dynamics of chimera states in heterogeneous {K}uramoto networks.
\newblock {\em Physica D}, 238(16):1569--1588, 2009.

\bibitem{Lamb1997}
J.~S.~W.~Lamb and J.~A.~G.~Roberts.
\newblock Time-reversal symmetry in dynamical systems: A survey.
\newblock {\em Physica D}, 112(1-2):1--39, 1998.

\bibitem{Oroszetal2009}
G.~Orosz, J.~Moehlis and P.~Ashwin.
\newblock Designing the Dynamics of Globally Coupled Oscillators.
\newblock {\em Prog. Theor. Phys.}, 122:611-630 2009.

\bibitem{Panaggio2015}
M.~J. Panaggio and D.~M. Abrams.
\newblock {Chimera states: coexistence of coherence and incoherence in networks
	of coupled oscillators}.
\newblock {\em Nonlinearity}, 28(3):R67--R87, 2015.

\bibitem{Pikovsky2003}
A. Pikovsky, M. Rosenblum and J. Kurths.
\newblock {\em {Synchronization: A Universal Concept in Nonlinear Sciences}}.
\newblock Cambridge University Press, 2003.

\bibitem{Pikovsky2006}
A.~Pikovsky and P.~Rosenau.
\newblock {Phase compactons}.
\newblock {\em Physica D}, 218(1):56--69, 2006.

\bibitem{SK86}
H.~Sakaguchi and Y.~Kuramoto.
\newblock A soluble active rotator model showing phase transitions via mutual
  entrainment.
\newblock {\em Prog. Theor. Phys.}, 76(3):576--581, 1986.

\bibitem{Sevryuk1986}
M.B. Sevryuk.
\newblock {\em Reversible Systems}.
\newblock Springer Lecture Notes in Mathematics. Springer, Berlin, 1986.

\bibitem{Strogatz2000}
S.~H. Strogatz.
\newblock {From Kuramoto to Crawford: exploring the onset of synchronization in
  populations of coupled oscillators}.
\newblock {\em Physica D}, 143(1-4):1--20, 2000.

\bibitem{Politi}
E.~Ullner and A.~Politi.
\newblock Self-Sustained Irregular Activity in an Ensemble of Neural Oscillators.
\newblock {\em Phys. Rev. X}, 6:011015, 2016.
 
\bibitem{Watanabe1994}
S. Watanabe and S.~H. Strogatz.
\newblock {Constants of motion for superconducting Josephson arrays}.
\newblock {\em Physica D}, 74(3-4):197--253, 1994.

\bibitem{Winfree}
A.T. Winfree.
\newblock {\em The geometry of biological time (second edition)}.
\newblock Springer, New York, 2001.

\end{thebibliography}

\appendix

\newpage

\section{Appendix: General cluster bifurcations for two harmonic coupling}
\label{app:twocluster}

As discussed in Section~\ref{sec:twocluster}, the bifurcation curves are tractable by finding the discriminant of the cubic that determines the location of the nontrivial $\Ss_p\times \Ss_{N-p}$ periodic orbits. More precisely, if $P=p/N$ and assuming that $q=-1$, then there is a bifurcation when (\ref{eq:twocluster3}) holds, i.e.
\[
0= a_0 + a_1 P + a_2 P^2 + a_3 P^3 +a_4 P^4,
\]
where the coefficients can be expressed as:
\begin{align*}
a_{4} & =64\sin\alpha(64r^{3}\sin^{3}\beta-48r^{2}\sin^{2}\beta\sin\alpha+12r\sin^{2}\alpha\sin\beta-\sin^{3}\alpha)\\
a_{3} & =128\sin\alpha(-64r^{3}\sin^{3}\beta+12r^{2}\sin\alpha\sin^{2}\beta-28r\sin^{2}\alpha\sin\beta+\sin^{3}\alpha)\\
a_{2} & =256\sin^{2}\beta\cos^{2}\beta r^{4}+256\sin\beta(20\sin\alpha\sin^{2}\beta+\sin\beta\cos\alpha\cos\beta+4\sin\alpha)r^{3}\\
& \qquad+64(62\sin^{2}\alpha\cos^{2}\beta-\sin(2\alpha)\sin\beta\cos\beta-73\sin^{2}\alpha+\sin^{2}\beta)r^{2}\\
& \qquad+64\sin\alpha(23\sin^{2}\alpha\sin\beta+7\sin\alpha\cos\alpha\cos\beta-5\sin\beta)r-32\sin^{2}\alpha(2\sin^{2}\alpha+1)\\
a_{1} & =-256r^{4}\sin^{2}\beta\cos^{2}\beta-256(4\sin\alpha\sin\beta(\sin^{2}\beta+1)+\sin^{2}\beta\cos\alpha\cos\beta)r^{3}\\
& \qquad+64(14\sin^{2}\alpha\sin^{2}\beta+11\sin^{2}\alpha-\sin^{2}\beta+\sin(2\alpha)\sin\beta\cos\beta)r^{2}\\
& \qquad+64(11\sin\alpha\sin\beta\cos^{2}\alpha-7\sin^{2}\alpha\cos\alpha\cos\beta-6\sin\beta\sin\alpha)r+32\sin^{2}\alpha\\
a_{0} & =64r^{4}\cos^{2}\beta+64\left(3\sin\alpha\sin\beta+\cos(\alpha-\beta)\right)r^{3}+16(1-13\sin^{2}\alpha+\cos^{2}\beta\\
& \qquad-2\cos\alpha\cos\beta\cos(\alpha-\beta))r^{2}+16((3-8\cos^{2}\alpha)\cos(\alpha-\beta)+4\cos\alpha\cos\beta)r-4
\end{align*}
It does not seem to be easy to express any of $r$, $\alpha$ or $\beta$ in terms of the other two, but this does make the bifurcation curves numerically computable by root finding (\ref{eq:twocluster3}). Note that this depends on $g$ being a trigonometric polynomial, though clearly this will become too complicated to work with for $g$ with many harmonics.

\end{document}